\documentclass[a4paper]{amsart}

\usepackage[alphabetic]{amsrefs}
\usepackage{braket,color,graphicx,tikz,hyperref,newtxtext,newtxmath}
\usepackage[headings]{fullpage}

\theoremstyle{plain}
\newtheorem{theorem}{Theorem}[section]
\newtheorem{lemma}[theorem]{Lemma}
\newtheorem{prop}[theorem]{Proposition}
\newtheorem{coro}[theorem]{Corollary}
\theoremstyle{definition}

\theoremstyle{remark}
\newtheorem{remark}[theorem]{Remark}
\numberwithin{equation}{section}

\newcommand{\abs}[1]{\lvert#1\rvert}

\newcommand{\Lr}[1]{\left(#1\right)}

\newcommand{\nm}[2]{\|\,#1\,\|_{#2}}

\newcommand{\na}{\nabla}

\usepackage{subfigure,comment}
\newtheorem{assumption}{Assumption}

\graphicspath{{Images/}}

\begin{document}
	\title[A concurrent global-local numerical method for multiscale parabolic equations]{A concurrent global-local numerical method for multiscale parabolic equations}
	\author[Y. Liao]{Yulei Liao}
	\address{Department of Mathematics, Faculty of Science, National University of Singapore, 10 Lower Kent Ridge Road, Singapore 119076, Singapore}
	\email{ylliao@nus.edu.sg}
	\author[Y. Liu\and P. Ming]{Yang Liu\and Pingbing Ming}
	\address{SKLMS, Institute of Computational Mathematics and Scientific/Engineering Computing, AMSS, Chinese Academy of Sciences, Beijing 100190, China}
	\address{School of Mathematical Sciences, University of Chinese Academy of Sciences, Beijing 100049, China}
	\email{liuyang2020@lsec.cc.ac.cn, mpb@lsec.cc.ac.cn}
	\thanks{This work is supported by funds from the National Key R\&D Program of China (\# 2024YFA1012502) and the National Natural Science Foundation of China (NSFC: \# 12371438).}
	\date{\today}
	
	\begin{abstract}
This paper presents a concurrent global-local numerical method for solving multiscale parabolic equations in divergence form. The proposed method employs hybrid coefficient to provide accurate macroscopic information while preserving essential microscopic details within specified local defects. Both the macroscopic and microscopic errors have been improved compared to existing results, eliminating the factor of $\Delta t^{-1/2}$ when the diffusion coefficient is time-independent. Numerical experiments demonstrate that the proposed method effectively captures both global and local solution behaviors.
\end{abstract}
\maketitle
\section{Introduction}\label{sec:intro}
Multiscale problems are pervasive in various scientific and engineering disciplines, characterized by phenomena occurring at multiple spatial and temporal scales~\cite{fish2009multiscale}. Within this broad field, multiscale parabolic equations are fundamental in modeling diffusion processes in heterogeneous media. These equations often involve diffusion coefficients that vary significantly across different regions, posing additional difficulties for numerical simulations. In the present work, we focus on a multiscale parabolic equation in divergence form with Dirichlet boundary condition:
	\begin{equation}
	\left\{
		\begin{aligned}
			\partial_t u^\varepsilon(x,t)-\nabla\cdot(a^\varepsilon(x)\nabla u^\varepsilon(x,t))&=f(x,t), &&(x,t)\in D\times(0,T)=:D_T, \\
			u^\varepsilon(x,0)&=u_0(x), && (x,t)\in D\times\{0\}, \\
			u^\varepsilon(x,t)&=0, &&(x,t)\in \partial D\times (0,T),
		\end{aligned}\right.
	\end{equation}
where $D\subset \mathbb{R}^n$ is a convex bounded domain and $\varepsilon>0$ is a small parameter that signifies the multiscale nature of the problem. The source term $f\in L^2(0,T;H^{-1}(D))$ and initial value $u_0\in H_0^1(D)$. The time-independent rough diffusion coefficient $a^\varepsilon(x)$ lies in $\mathcal{M}(\lambda, \Lambda; D)$ for some $\lambda$, $\Lambda>0$, where
	\[
	\begin{aligned}
		\mathcal{M}(\lambda, \Lambda ; D):=\Big\{a \in\left[L^{\infty}(D)\right]^{n \times n} &\mid\xi \cdot a(x) \xi \geq \lambda|\xi|^2, \xi \cdot a(x) \xi\ge \frac{1}{\Lambda}|a(x) \xi|^2 \\
		& \text { for any } \xi \in \mathbb{R}^n \text { and a.e. } x\in D\Big\}.
	\end{aligned}
	\]

On the analytic level, the homogenization theory, particularly H-convergence, provides a powerful framework for approximating the behavior of such multiscale systems by averaging the effects of small-scale variations. For a sequence $\{a^\varepsilon(x)\}_{\varepsilon>0} \in \mathcal{M}(\lambda, \Lambda ; D)$, there exists a subsequence (still denoted by $\{a^\varepsilon(x)\}_{\varepsilon>0}$ for convenience) and $A(x) \in\mathcal{M}(\lambda, \Lambda ; D)$ such that $a^\varepsilon(x)$ H-converges to $A(x)$ \cite{tartar_general_2010}. Set $u(x,t)$ the solution to the corresponding homogenization problem
\begin{equation}\label{eq:u}\left\{
	\begin{aligned}
		\partial_t u(x,t)-\nabla\cdot(A(x)\nabla u(x,t))&=f(x,t), &&(x,t)\in D_T, \\
		u(x,0)&=u_0(x), && (x,t)\in D\times\{0\}, \\
		u(x,t)&=0, &&(x,t)\in \partial D\times (0,T).
	\end{aligned}\right. 
	\end{equation}
According to~\cite{bensoussan_asymptotic_2011, 1992Correctors}, the solutions $u^{\varepsilon}$ and $u$ exist in $L^2(0, T ; H_0^1(D))$, and satisfy
\[
\left\{\begin{aligned}
		u^{\varepsilon} &\rightharpoonup u &&\text {weakly in } L^2(0, T ; H_0^1(D)),\\
		a^\varepsilon\nabla u^{\varepsilon} &\rightharpoonup A\nabla u &&\text {weakly in } [L^2(0, T; L^2(D))]^d,\\
		\partial_t u^{\varepsilon} &\rightharpoonup \partial_t u &&\text {weakly in }L^2(0, T ; H^{-1}(D)).
	\end{aligned}\right.
\]
	
Numerous numerical methods founded on the idea of homogenization have been suggested and extensively examined in existing studies such as multiscale finite element method~\cite{MultiscaleFiniteElement1997} and heterogeneous multiscale method (HMM)~\cite{eHeterognousMultiscaleMethods2003, AbdulleE:2003,chenHeterogeneousMultiscaleMethod2005}. HMM aims to capture the macroscopic behavior of the equation without microscopic details. It consists of the macroscopic solver to approximate the homogenization solution $u$, which rely on the effective matrix $A$ and the microscopic solver to approximate this effective matrix. Review on HMM may be found in~\cite{Engquist:2007} and~\cite{Abdulle:2012}. 

In this work, we shall employ an HMM type method to capture the microscopic information in a small defect $K_0\subset D$ and macroscopic information outside the region. For such scenario, there are two types of numerical methods. The first one is based on domain decomposition~\cite{apoung2007numerical,hu2000gradient,gervasio2001heterogeneous,glowinski2004multi,lozinski2011numerical}, particularly the Arlequin method~\cite{Dhia:2005arlequin,Albella:2019mathematical,Gorynina:2021some}. These methods apply different operators to different subdomains of the computational domain. The second one is the global-local method~\cite{oden2000estimation, oden2001estimation}, which compute the homogenized solution $u$ firstly, and then exploit it as the boundary condition to approximate the multiscale solution $u^\varepsilon$ on a local region. 
Following~\cite{huang_concurrent_2018,mingNitscheHybridMultiscale2022} several concurrent global-local methods have been proposed for multiscale elliptic equations with matching and non-matching grids, respectively. The key idea is to hybridize the microscopic and macroscopic coefficients by a function as $b^\varepsilon(x)=\rho(x) a^\varepsilon(x)+(1-\rho(x))A(x)$, where the transition function $\rho=1$ in $K_0$ and vanishes away from $K_0$ without any smoothness requirement.
	
In this work, we extend the concurrent global-local method proposed in~\cite{huang_concurrent_2018} to the multiscale parabolic problems. Specifically, we plan to study the microscopic information in a subdomain $K_0 \subset D$, with $\text{diam }K_0=d$. We slightly extend $K_0$, i.e., set $K_0\subset K$, with $\text{dist}(K_0,\partial K\backslash \partial D)=cd,$ where $c$ is a positive universal constant. The hybrid coefficient $b^\varepsilon(x)=\rho(x) a^\varepsilon(x)+(1-\rho(x))A(x)$, where $\rho(x)$ is a transition function which satisfies 
\[
\left\{
\begin{aligned}
	\rho(x)&=1, && x\in K_0,\\
	0\leq\rho(x)&\leq 1, && x\in K,\\	
	\rho(x)&=0, && x\in D\backslash K.
\end{aligned}
\right.
\]
An interesting finding in Theorem~\ref{thm:local} suggests that the pollution of the local energy error may be reduced if we take $\rho$ as the indicator function of $K$.

We use the backward Euler scheme as the temporal discretization with a uniform time step $\Delta t$ and the Lagrange linear finite element to approximate the solution space with local mesh size $h$ over $K$ and global mesh size $H$ over $D$. Specifically, let $\mathcal{T}_{h,H}$ be a shape-regular triangulation of $D$ in the sense of~\cite{ciarlet_finite_2002}, which satisfies 
\[
h=\max_{T\in\mathcal{T}_{h,H},T\cap K\neq \emptyset} h_T \quad\text{and}\quad H=\max_{T\in\mathcal{T}_{h,H}} h_T,
\]
and the Lagrange linear finite element space is defined as
\[
X_{h,H}:= \{v\in H_0^1(D)\mid v|_T\in \mathbb{P}_1(T), \quad\text{for all}\quad T\in \mathcal{T}_{h,H}\}.  
\]
Generally we do not have the information of the effective matrix $A$. In practice, the hybrid coefficient $b_H^\varepsilon(x)=\rho(x) a^\varepsilon(x)+(1-\rho(x))A_H(x)$, where $A_H$ is the approximation of $A$ obtained by HMM or other types of numerical homogenization methods. In the framework of HMM, we solve several cell problems over cubes with length $\delta_0$ to approximate $A_H$, and use $e(\mathrm{HMM})$ measuring error caused by such approximation; e.g. for local periodic $a^\varepsilon(x)=a(x,x/\varepsilon)$ with $a(x,y)$ periodic in $y$, it is consistent with the elliptic cases~\cite{Ming:2005,Du:2010} that $e(\mathrm{HMM})$ is $\mathcal{O}(\delta_0+\varepsilon/\delta_0)$.  Denote $t_k=k\Delta t$, $f_k(x)=f(x,t_k)$, $u_k(x)=u(x,t_k)$ and $u^\varepsilon_k(x)=u^\varepsilon(x,t_k)$. 
To numerically solve the equation, we find $U_k\in X_{h,H},k=1,\dots,m$ such that $t_m=T$ and
\begin{equation}
	\label{test}
	(\delta U_k,V)_D +(b^\varepsilon_H\nabla U_k,\nabla V)_D=( f_k,V)_D\qquad\text{for all}\quad V\in X_{h,H},
\end{equation}
where $\delta U_k=(U_k-U_{k-1})/\Delta t$, and the initial value $U_0\in X_{h,H}$ is the solution of the following variational problem such that
\begin{equation}\label{eq:U0}
(b^\varepsilon_H\nabla U_0,\nabla V)_D=(b^\varepsilon_H\nabla u_0,\nabla V)_D\qquad\text{for all}\quad V\in X_{h,H}.
\end{equation}

To judge whether $U_m$ can successfully capture the global information $u$ over $D$ and the local information $u^\varepsilon$ over the local region $K_0$, we shall analyze this concurrent scheme. Unlike the multiscale elliptic equations~\cite{huang_concurrent_2018}, time evolution presents further analytical difficulties. As shown in~\cite{mingAnalysisHeterogeneousMultiscale2007,Eckhardt:2023fully}, the HMM approximates the homogenization solution $u$ for multiscale parabolic equations with the $H^1$-norm error estimate $\mathcal{O}(\Delta t+H+e(\mathrm{HMM})\Delta t^{-1/2})$. The factor $\Delta t^{-1/2}$ is frustrating since it indicates that the time step can not be arbitrarily small. In Theorem~\ref{thm:global}, we eliminate this factor for time-independent coefficient $a^\varepsilon$, and the global $H^1$-error reads as
\[
\nm{u_m-U_m}{H^1(D)}\le C(\lambda,\Lambda,T,u)(\Delta t+H+\eta(K)+e(\mathrm{HMM})),
\]
where $\eta(K)$ measures the pollution caused by the local region $K$, which also appears in multiscale elliptic equations~\cite{huang_concurrent_2018}. We achieve this improvement with the aid of a Galerkin projection operator~\eqref{eq:glocalProj} onto the finite element space for an auxiliary elliptic problem, not the projection for original parabolic equations~\eqref{eq:u}  in~\cite{mingAnalysisHeterogeneousMultiscale2007,Eckhardt:2023fully}.

Since we only concentrate on the microscopic information on a local region $K_0$, we expect the convergence of this concurrent scheme to the multiscale solution $u^\varepsilon$ in an interior energy norm. Studies for interior energy error estimates can be traced back to~\cite{nitsche_interior_1974,Bramble:1975maximum} for approximating second order elliptic equations on quasi-uniform meshes. Later, similar results have been proven on shape-regular meshes in~\cite{demlow_local_2011} using a novel superapproximation result. They found that the interior $H^1$-error may be bounded by a local approximation term plus a global pollution term, like $\mathcal{O}(h+H^2/d)$. These results are useful to analyze several local and parallel algorithms such as two-grid methods for elliptic type equations~\cite{jinchao2001local,Wang:2023local}. 

For parabolic equations, we discuss interior energy errors in space $L^2(0,T;H^1(K_0))$, or the discrete corresponding errors in space $\ell^2(0,T;H^1(K_0))$, with
$$
\|f\|_{L^p(0, T; X)}^p:=\int_0^T\|f(\cdot, t)\|_X^p d t \quad \text { and } \quad\|f\|_{\ell^p(0, T; X)}^p:=\Delta t \sum_{k=1}^m\|f_k\|_X^p,
$$
for a Banach space $X$.
Thom{\'e}e et al.~\cite{Bramble:1977some,Thomee:1979some,Schatz:1998stability} have extended Nitsche's work~\cite{nitsche_interior_1974,Bramble:1975maximum} for parabolic equations with semi-discrete schemes.  Unfortunately, these works are not suitable for the multiscale problems, because they require the smoothness of the coefficient $a^\varepsilon$ and the solution $u^\varepsilon$.
We also note that the local $H^1$-error for the two-grid algorithm in~\cite{jinchao2001local} for transient Stokes equations~\cite{shang_local_2010} reads as $\mathcal{O}(\Delta t^{-1/2}(\Delta t+h+H^2))$. All these works hangs on  the interior energy stability as a bridge. We follow another way to derive the interior energy estimate by considering $\partial_t u^\varepsilon$ as an extra source term of the elliptic equation; see Theorem~\ref{thm:local} and Remark~\ref{rmk:local} that
\begin{align*}
\nm{\nabla(u^\varepsilon-U)}{\ell^2(0,T;L^2(K_0))}&\le C(\lambda,\Lambda,T,u^\varepsilon,u)\Bigl(\inf_{V\in[X_{h,H}]^m}\nm{\nabla(u^\varepsilon-V)}{L^2(K)}\\
&\qquad+d^{-1}(|K|^{1/2-1/p}+\Delta t+H^2+\eta(K)+e(\mathrm{HMM}))\Bigr),
\end{align*}
with Meyer's regularity exponent $p>2$~\cite{bensoussan_asymptotic_2011}. Hence, only the Meyer's regularity for $u^\varepsilon$ required and we eliminate the frustrating factor $\Delta t^{-1/2}$ for the interior energy error over $K_0$.

The rest of the paper is organized as follows. In Section \ref{sec:Hconver}, we study the H-limit of the mixing of the general time-dependent hybrid coefficient for parabolic equations, which may be of independent interest. In Section \ref{sec:error}, we present the accuracy for retrieving the macroscopic and microscopic information, which are the main technique part of this work. In Section \ref{sec:numer}, we provide several numerical experiments which are adopted from~\cite{huang_concurrent_2018,mingNitscheHybridMultiscale2022}, fitted theoretical predictions. Finally, Section \ref{sec:concl} concludes the paper with a summary and potential future research directions. 

Throughout this paper, we shall use standard Sobolev spaces $W^{m,p}(\Omega)$ for any bounded domain $\Omega\subset\mathbb R^n$. Without further explanation is provided, the constant $C$ only depends on $\lambda,\Lambda$, the end time $T$ and the domain $D$, and is independent of $\varepsilon$, the diam of the local region $d$, time step $\Delta t$ and the mesh sizes $h,H$.

\section{H-convergence for parabolic equations}\label{sec:Hconver}
Before analyzing the convergence of the method, we firstly examine the homogenization of the mixing of coefficients, following the elliptic case in~\cite[\S~2]{huang_concurrent_2018}. Under H-convergence, a time-independent coefficient yields the same effective limit as in the elliptic problem, which suffices for our purposes. Here, we slightly extend the setting and study the case of time-dependent coefficients. We keep the notation unchanged, as this does not constitute an abuse of notation and should not cause confusion. Specifically, if $a^\varepsilon(x,t)\in \mathcal{M}(\lambda,\Lambda;D_T)$, where $D_T=D\times[0,T]$ and
\[
\begin{aligned}
	\mathcal{M}(\lambda, \Lambda ; D_T):=\Set{a \in\left[L^{\infty}(D_T)\right]^{n \times n}|
	a(\cdot,t)\in \mathcal{M}(\lambda, \Lambda ; D)  \text { for all } t\in [0,T]},
\end{aligned}
\]
then by H-convergence theory, there exists a subsequence of the sequence $a^\varepsilon(x,t)$ that H-converge \cite{zhikovGconvergenceParabolicOperators1981,dallaglioCorrectorResultConverging1997} to certain $A(x,t)\in\mathcal{M}(\lambda,\Lambda;D_T)$. In continuous sense, let the hybrid coefficient $b^\varepsilon(x,t)=\rho(x,t) a^\varepsilon(x,t)+(1-\rho(x,t))A(x,t)$ for some cut-off function $\rho(x,t):\Omega\times[0,T]\to[0,1]$. Similarly, $b^\varepsilon(x,t)\in \mathcal{M}(\lambda,\Lambda;D_T)$ and there exists a subsequence of $b^\varepsilon$ that H-converge to certain $B(x,t)\in\mathcal{M}(\lambda,\Lambda;D_T)$.
 
In this part, we measure the difference between $A(x,t)$ and $B(x,t)$. For the sake of convenience, we still denote the subsequences by  $a^\varepsilon(x,t)$ and $b^\varepsilon(x,t)$. Denote the $L^2(D_T)$ inner product by $(\cdot,\cdot)_{D_T}$, and the duality pairing between $L^2(0,T;H^{-1}(D))$ and $L^2(0,T;H_0^1(D))$ by $\langle \cdot,\cdot\rangle_{D_T}$. 
\begin{prop}
	There holds
	\begin{equation}
		\label{Hconverge}
		|A(x,t)-B(x,t)| \leq 2 \Lambda(\Lambda / \lambda+\sqrt{\Lambda / \lambda}) \rho(x,t)(1-\rho(x,t)) \text { a.e. } (x,t) \in D_T \text {. }
	\end{equation}
Here $|\cdot|$ is the $\ell^2$-norm of a matrix.
\end{prop}

The proof is essentially a combination of  ~\cite[Theorem 1]{antonicParabolicHconvergenceSmallamplitude2009} and \cite[Proposition 2.1]{huang_concurrent_2018}. 

\begin{proof}
Define 
	\[
	\mathcal{W}:=\Set{v\in L^2(0,T;H_0^1(D))|\partial_t v\in L^2(0,T;H^{-1}(D))},
	\]
and subspaces
	\[
	\mathcal{W}_0:=\Set{v\in \mathcal{W}|v(x,0)=0},\qquad\mathcal{W}_T:=\Set{v\in \mathcal{W}|v(x,T)=0}.\]
	For given $\phi_0 \in 
	\mathcal{W}_0$ and $\psi_0\in\mathcal{W}_T$, we define $\phi^\varepsilon\in\mathcal{W}_0$ as the unique solution of the initial-boundary value problem
	\begin{equation}
		\label{Hcon1}
		\left\{
		\begin{aligned}
			\partial_t \phi^\varepsilon-\nabla\cdot(a^\varepsilon \nabla \phi^\varepsilon)&=\partial_t \phi_0-\nabla\cdot(A\nabla \phi_0),\quad \text{in~} D_T, \\
	\phi^\varepsilon&\in\mathcal{W}_0,
		\end{aligned}\right.
	\end{equation}
and $\psi^\varepsilon\in\mathcal{W}_T$ as the unique solution of the backward-in-time problem
\begin{equation}
		\label{Hcon2}
		\left\{
		\begin{aligned}
			-\partial_t \psi^\varepsilon-\nabla\cdot((b^\varepsilon)^{\top} \nabla \psi^\varepsilon)&=-\partial_t \psi_0-\nabla\cdot(B^\top\nabla \psi_0), \quad\text{in~} D_T, \\			\psi^\varepsilon&\in\mathcal{W}_T.
		\end{aligned}\right.
	\end{equation}
Here $(b^\varepsilon)^{\top}$ is the transpose of the matrix $b^\varepsilon$. By H-convergence theory, we have
	\begin{align*}
	\phi^{\varepsilon} \rightharpoonup \phi_0,\quad
	&\psi^{\varepsilon} \rightharpoonup \psi_0
	 &&\text {weakly in } L^2(0, T ; H_0^1(D)),\\
	a^\varepsilon\nabla \phi^{\varepsilon} \rightharpoonup A\nabla \phi_0 ,\quad
	&(b^\varepsilon)^\top\nabla \psi^{\varepsilon} \rightharpoonup B^\top\nabla \psi_0
	&&\text {weakly in } [L^2(0, T; L^2(D))]^d,\\
	\partial_t \phi^{\varepsilon} \rightharpoonup \partial_t \phi_0 ,\quad
	&\partial_t \psi^{\varepsilon} \rightharpoonup \partial_t \psi_0
	 &&\text {weakly in }L^2(0, T ; H^{-1}(D)).
\end{align*}
	
After multiplying the above equations by $\varphi \psi^\varepsilon$ and $\varphi \phi^\varepsilon$ respectively, for some $\varphi \in \mathrm{C}_c^{\infty}(D_T)$, we obtain
\[
	\left\langle\partial_t \phi^\varepsilon, \varphi \psi^\varepsilon\right\rangle_{D_T}+(a^\varepsilon \nabla \phi^\varepsilon, \nabla\left(\varphi \psi^\varepsilon\right))_{D_T}=\left\langle\partial_t \phi_0, \varphi \psi^\varepsilon\right\rangle_{D_T}+(A \nabla \phi_0, \nabla(\varphi \psi^\varepsilon))_{D_T}
\]
and
	$$
	-\left\langle\partial_t \psi^\varepsilon, \varphi \phi^\varepsilon\right\rangle_{D_T}+ ((b^\varepsilon)^\top \nabla \psi^\varepsilon, \nabla\left(\varphi \phi^\varepsilon\right))_{D_T}=-\left\langle\partial_t \psi_0, \varphi \phi^\varepsilon\right\rangle_{D_T}+(B^\top \nabla \psi_0, \nabla\left(\varphi \phi^\varepsilon\right))_{D_T} .
	$$
	By subtracting these equalities, after using the identity $\left\langle\partial_t \phi^\varepsilon, \varphi\psi^\varepsilon\right\rangle_{D_T}=-\langle\psi^\varepsilon\partial_t\varphi+\varphi\partial_t\psi^\varepsilon,\phi^\varepsilon\rangle_{D_T}$ and $\left\langle\partial_t \phi_0, \varphi\psi^\varepsilon\right\rangle_{D_T}=-\left\langle\psi^\varepsilon\partial_t\varphi+\varphi\partial_t\psi^\varepsilon,\phi_0\right\rangle_{D_T}$ we have
	\begin{align*}
		& -(\phi^\varepsilon,\psi^\varepsilon \partial_t \varphi)_{D_T}+(a^\varepsilon \nabla \phi^\varepsilon ,\psi^\varepsilon \nabla \varphi+\varphi \nabla \psi^\varepsilon)_{D_T}
		-((b^\varepsilon)^\top \nabla \psi^\varepsilon,\phi^\varepsilon \nabla \varphi+\varphi \nabla \phi^\varepsilon)_{D_T}  \\
		& \qquad=\left\langle\partial_t \psi_0, \varphi \phi^\varepsilon\right\rangle_{D_T}-\left\langle\partial_t \psi^\varepsilon, \varphi \phi_0\right\rangle_{D_T}-(\phi_0,\psi^\varepsilon \partial_t \varphi)_{D_T}\\
		&\qquad\qquad+ (A \nabla \phi_0 , \psi^\varepsilon \nabla \varphi+\varphi \nabla \psi^\varepsilon )_{D_T}- (B^\top \nabla \psi_0,\phi^\varepsilon \nabla \varphi+\varphi \nabla \phi^\varepsilon)_{D_T}.
	\end{align*}
	Now we can pass to the limit as $\varepsilon\rightarrow 0$; the only terms left after the cancellation are
	$$
	\lim_{\varepsilon\rightarrow 0}(\varphi(a^\varepsilon-b^\varepsilon)\nabla\phi^\varepsilon,\nabla \psi^\varepsilon)_{D_T}=(\varphi(A-B)\nabla \phi_0, \nabla \psi_0)_{D_T}.
	$$
	
	Let $\varphi \geq 0$, and we define
	$$
	X:=\lim _{\varepsilon \rightarrow 0}\left(\varphi\left(b^{\varepsilon}-a^{\varepsilon}\right) \nabla \phi^{\varepsilon}, \nabla \psi^{\varepsilon}\right)_{D_T}=(\varphi(B-A)\nabla \phi_0, \nabla \psi_0)_{D_T}.
	$$
	It is clear that
	\[
		X  \leq \limsup _{\varepsilon \rightarrow 0}\left(\varphi(1-\rho)\left|\left(A-a^{\varepsilon}\right) \nabla \phi^{\varepsilon}\right|,\left|\nabla \psi^{\varepsilon}\right|\right)_{D_T} \leq 2 \Lambda \limsup _{\varepsilon \rightarrow 0}\left(\varphi(1-\rho)\left|\nabla \phi^{\varepsilon}\right|,\left|\nabla \psi^{\varepsilon}\right|\right)_{D_T}.
	\]
	For any $\alpha>0$, we bound $X$ as
	
	\begin{equation}\label{eq:X}
	\begin{aligned}
		X & \leq \frac{2 \Lambda}{\lambda}\left(\alpha \lambda \limsup _{\varepsilon \rightarrow 0}(\varphi(1-\rho),\left|\nabla \phi^{\varepsilon}\right|^2)_{D_T}+\frac{\lambda}{4 \alpha} \limsup _{\varepsilon \rightarrow 0}(\varphi(1-\rho),\left|\nabla \psi^{\varepsilon}\right|^2)_{D_T}\right) \\
		& \leq \frac{2 \Lambda}{\lambda}\left(\alpha \limsup _{\varepsilon \rightarrow 0}\left(\varphi(1-\rho) a^{\varepsilon} \nabla \phi^{\varepsilon}, \nabla \phi^{\varepsilon}\right)_{D_T}+\frac{1}{4 \alpha} \limsup _{\varepsilon \rightarrow 0}\left(\varphi(1-\rho) b^{\varepsilon} \nabla \psi^{\varepsilon}, \nabla \psi^{\varepsilon}\right)_{D_T}\right). \\
		\end{aligned}
		\end{equation}
	After multiplying \eqref{Hcon1} by $(1-\rho)\varphi \phi^\varepsilon$, we obtain
	\begin{align*}
	\left\langle\partial_t \phi^\varepsilon, (1-\rho)\varphi \phi^\varepsilon\right\rangle_{D_T}+(a^\varepsilon \nabla \phi^\varepsilon, \nabla[(1-\rho)\varphi \phi^\varepsilon])_{D_T}&=\left\langle\partial_t \phi_0, (1-\rho)\varphi \phi^\varepsilon\right\rangle_{D_T}\\
	&\qquad+(A \nabla \phi_0, \nabla[(1-\rho)\varphi \phi^\varepsilon])_{D_T}.
	\end{align*}
	We pass to the limit as $\varepsilon\rightarrow 0$, and obtain
	$$
	\lim_{\varepsilon \rightarrow 0}(\varphi(1-\rho) a^{\varepsilon} \nabla \phi^{\varepsilon}, \nabla \phi^{\varepsilon})_{D_T}=(\varphi(1-\rho) A \nabla \phi_0, \nabla \phi_0)_{D_T}.
	$$
	Similarly, we obtain
		$$
	\lim_{\varepsilon \rightarrow 0}(\varphi(1-\rho) b^{\varepsilon} \nabla \psi^{\varepsilon}, \nabla \psi^{\varepsilon})_{D_T}=(\varphi(1-\rho) B \nabla \psi_0, \nabla \psi_0)_{D_T}.
	$$
	
	Substituting above two limits into~\eqref{eq:X}, we obtain
		$$
		\begin{aligned}
		X & \leq \frac{2 \Lambda}{\lambda}\left(\alpha\left(\varphi(1-\rho) A \nabla \phi_0, \nabla \phi_0\right)_{D_T}+\frac{1}{4 \alpha}\left(\varphi(1-\rho) B \nabla \psi_0, \nabla \psi_0\right)_{D_T}\right) \\
		& \leq \frac{2 \Lambda^2}{\lambda}\left(\alpha(\varphi(1-\rho),\left|\nabla \phi_0\right|^2)_{D_T}+\frac{1}{4 \alpha}(\varphi(1-\rho),\left|\nabla \psi_0\right|^2)_{D_T}\right),
	\end{aligned}
	$$
	which implies that,
	$$
	\left|(B-A) \nabla \phi_0 \cdot \nabla \psi_0\right| \leq \frac{2 \Lambda^2}{\lambda}(1-\rho)\left(\alpha\left|\nabla \phi_0\right|^2+\frac{1}{4 \alpha}\left|\nabla \psi_0\right|^2\right),\qquad\text{a.e.}\quad (x,t) \in D_T,
	$$
	due to $\varphi$ is arbitrary. Optimizing $\alpha$, we obtain that
	$$
	\left|(B-A) \nabla \phi_0 \cdot \nabla \psi_0\right| \leq \frac{2 \Lambda^2}{\lambda}(1-\rho)\left|\nabla \phi_0\right|\left|\nabla \psi_0\right|,\qquad\text{a.e.}\quad (x,t) \in D_T.
	$$

    Finally by the definition of the $\ell^2$-norm of a matrix, it gives
	\begin{equation}
		\label{H1mrho}
		|A(x,t)-B(x,t)| \leq \frac{2 \Lambda^2}{\lambda}(1-\rho(x,t)) \quad \text { a.e. } (x,t) \in D_T,
	\end{equation}
	since $\phi_0$ and $\psi_0$ are arbitrary.
	
	Next, we prove
	\begin{equation}
		\label{Hrho}
		|A(x,t)-B(x,t)| \leq 2 \Lambda \sqrt{\Lambda / \lambda} \rho(x,t) \quad \text { a.e. } (x,t) \in D_T \text {. }
	\end{equation}
	The proof of \eqref{Hrho} is similar with the one that leads to \eqref{H1mrho}. Motivated by $b^{\varepsilon}-A = \rho(a^{\varepsilon}-A)$, we define
	$$
	Y:=\lim _{\varepsilon \rightarrow 0}\left(\varphi\left(b^{\varepsilon}-A\right) \nabla \phi_0, \nabla \psi^{\varepsilon}\right)_{D_T} =(\varphi(B-A)\nabla\phi_0,\nabla\psi_0)_{D_T} \quad \text { for } \quad \varphi \geq 0.
	$$
	It is clear that
	$$
	Y \leq \limsup _{\varepsilon \rightarrow 0}\left(\varphi \rho\left|\left(A-a^{\varepsilon}\right) \nabla \phi_0\right|,\left|\nabla \psi^{\varepsilon}\right|\right)_{D_T} \leq 2 \Lambda \limsup _{\varepsilon \rightarrow 0}\left(\varphi \rho\left|\nabla \phi_0\right|,\left|\nabla \psi^{\varepsilon}\right|\right)_{D_T}.
	$$
	For any $\alpha>0$, we bound $Y$ as
	$$
	\begin{aligned}
	Y & \leq 2 \Lambda\left(\alpha(\varphi \rho,\left|\nabla \phi_0\right|^2)_{D_T}+\frac{1}{4 \alpha} \limsup _{\varepsilon \rightarrow 0}(\varphi \rho,\left|\nabla \psi^{\varepsilon}\right|^2)_{D_T}\right) \\
		& \leq 2 \Lambda\left(\alpha(\varphi \rho,\left|\nabla \phi_0\right|^2)_{D_T}+\frac{1}{4 \alpha \lambda} \limsup _{\varepsilon \rightarrow 0}\left(\varphi \rho b^{\varepsilon} \nabla \psi^{\varepsilon}, \nabla \psi^{\varepsilon}\right)_{D_T}\right)\\
		& \leq 2 \Lambda\left(\alpha\left(\varphi \rho, |\nabla \phi_0|^2\right)_{D_T}+\frac{1}{4 \alpha \lambda}\left(\varphi \rho B \nabla \psi_0, \nabla \psi_0\right)_{D_T}\right) \\
		& \leq 2 \Lambda\left(\alpha(\varphi \rho,\left|\nabla \phi_0\right|^2)_{D_T}+\frac{\Lambda}{4 \alpha \lambda}(\varphi \rho,\left|\nabla \psi_0\right|^2)_{D_T}\right),
	\end{aligned}
	$$
	which implies that
	$$
	\left|(B-A) \nabla \phi_0 \cdot \nabla \psi_0\right| \leq 2 \Lambda \rho\left(\alpha\left|\nabla \phi_0\right|^2+\frac{\Lambda}{4 \alpha \lambda}\left|\nabla \psi_0\right|^2\right),  \quad \text { a.e. } (x,t) \in D_T.
	$$
	Optimizing $\alpha$, we obtain
	$$
	\left|(B-A) \nabla \phi_0 \cdot \nabla \psi_0\right| \leq 2 \Lambda \sqrt{\Lambda / \lambda} \rho\left|\nabla \phi_0\right|\left|\nabla \psi_0\right| \quad \text { a.e. } (x,t) \in D_T.
	$$
	Thus we obtain \eqref{Hrho}.
	Finally we use the convex combination of \eqref{H1mrho} and \eqref{Hrho} as
	$$
	|A(x,t)-B(x,t)|=\rho(x)|A(x,t)-B(x,t)|+(1-\rho(x))|A(x,t)-B(x,t)|,
	$$
	this leads to \eqref{Hconverge}.
\end{proof}
If we replace $A$ by $A_H$ in the definition of $b^\varepsilon$, We may slightly generalize the above theorem as
\begin{coro} Let $A_H \in \mathcal{M}\left(\lambda^{\prime}, \Lambda^{\prime} ; D_T\right)$, and we define $b_H^{\varepsilon}=\rho a^{\varepsilon}+(1-\rho) A_H$. Denote by $B_H$ the H-limit of $b^{\varepsilon}_H$, then for a.e. $(x,t) \in D_T$, there holds
	\[
	|B_H(x,t)-\rho(x,t)A(x,t)-(1-\rho(x,t)) A_H(x,t)|\leq(\Lambda+\tilde{\Lambda}) \sqrt{\tilde{\Lambda} / \tilde{\lambda}}(\sqrt{\Lambda / \lambda}+1) \rho(x,t)(1-\rho(x,t)),
	\]
	where $\tilde{\Lambda}=\Lambda \vee \Lambda^{\prime}$ and $\tilde{\lambda}=\lambda \wedge \lambda^{\prime}$.
\end{coro}

\begin{remark}
If the transition function $\rho=\chi_K$; i.e. the indicator function of $K$, then the H-limit of the concurrent coefficient $b^\varepsilon$ is exactly the effective matrix $A$.
\end{remark}
\section{Error estimates}\label{sec:error}
In this part, we derive the error bounds for the proposed method, including the global error between $U$ and $u$ over domain $D$, and the interior energy error between $U$ and $u^\varepsilon$ over the defect domain $K_0$.
\subsection{Accuracy for retrieving the macroscopic information}
We firstly present the well-known stability results~\cite{Thomee:2007galerkin} of the scheme~\eqref{test}. The proof is included here for the sake of completeness.
\begin{lemma}
If $U=\{U_k\}_{k=1}^m$ is the solution of \eqref{test}, then
\begin{equation}\label{sta1}
\Vert U_m\Vert_{L^2(D)} \leq \Vert U_0\Vert_{L^2(D)}+C \|f\|_{\ell^2(0,T;H^{-1}(D))},
	\end{equation}
and
\begin{equation}\label{sta2}
\Vert \delta U \Vert_{\ell^2(0,T;L^2(D))}+ \Vert \nabla U_m\Vert_{L^2(D)} \leq C\left( \Vert \nabla U_0 \Vert_{L^2(D)} +  \|f\|_{\ell^2(0,T;L^2(D))}\right).
	\end{equation}
\end{lemma}

\begin{proof}
	Let $V=U_k$ in \eqref{test}, we have
	$$
	(U_k-U_{k-1},U_k)_D+\Delta t (b^\varepsilon_H\nabla U_k,\nabla U_k)_D=\Delta t (f_k,U_k)_D.
	$$
	By the identity $a(a-b)=\frac{1}{2}(a-b)^2+\frac{1}{2}(a^2-b^2)$, we obtain
\[
	\frac{1}{2}\Vert U_k-U_{k-1} \Vert_{L^2(D)}^2+\frac{1}{2}\Vert U_k \Vert_{L^2(D)}^2-
	\frac{1}{2}\Vert U_{k-1} \Vert_{L^2(D)}^2+\Delta t (b^\varepsilon_H\nabla U_k,\nabla U_k)_D=\Delta t (f_k,U_k)_D.
\]
There exists $C$ depends on $D$ such that
\begin{align*}
\frac{1}{2}\Vert U_k \Vert_{L^2(D)}^2  +\lambda\Delta t  \Vert\nabla U_k\Vert_{L^2(D)}^2 &\leq \frac{1}{2}\Vert U_{k-1} \Vert_{L^2(D)}^2+C\Delta t \Vert f_k\Vert_{H^{-1}(D)} \Vert \nabla U_k \Vert_{L^2(D)}\\
		&\leq\frac{1}{2}\Vert U_{k-1} \Vert_{L^2(D)}^2+\frac{\lambda}{2}\Delta t\Vert \nabla U_k \Vert_{L^2(D)}^2+C\Delta t \Vert f_k\Vert_{H^{-1}(D)}^2,
	\end{align*}
	i.e.,
	$$
	\Vert U_k \Vert_{L^2(D)}^2 + \lambda\Delta t \Vert\nabla U_k\Vert_{L^2(D)}^2\leq\Vert U_{k-1} \Vert_{L^2(D)}^2+C \Delta t \Vert f_k\Vert_{H^{-1}(D)}^2.
	$$
Summing up the above inequality from $1$ to $m$, we obtain~\eqref{sta1}.
	
	Let $V=\delta U_k$ in \eqref{test}, we have
	$$
	\Vert \delta U_k \Vert_{L^2(D)}^2 + (b^\varepsilon_H\nabla U_k,\nabla \delta U_k)_D=(f_k,\delta U_k)_D.
	$$
	By the identity $a(a-b)=\frac{1}{2}(a-b)^2+\frac{1}{2}(a^2-b^2)$, we obtain
	\begin{align*}
		\Delta t\Vert \delta U_k \Vert_{L^2(D)}^2 +\frac{1}{2}(b^\varepsilon_H\nabla U_k, \nabla U_k)_D &\leq\frac{1}{2}(b^\varepsilon_H\nabla U_{k-1}, \nabla U_{k-1})_D+\Delta t \| f_k  \|_{L^2(D)}\Vert \delta U_k \Vert_{L^2(D)} \\
		&\leq \frac 1 2\left((b^\varepsilon_H\nabla U_{k-1}, \nabla U_{k-1})_D+\Delta t\| f_k  \|_{L^2(D)}^2+\Delta t\|\delta U_k \Vert_{L^2(D)}^2\right).
	\end{align*}
Henceforth,
\[
	\Delta t\Vert \delta U_k \Vert_{L^2(D)}^2 +	 (b^\varepsilon_H\nabla U_k, \nabla U_k)_D \leq (b^\varepsilon_H\nabla U_{k-1}, \nabla U_{k-1})_D+\Delta t\| f_k  \|_{L^2(D)}^2.
\]
Since $b^\varepsilon_H$ is time-independent, summing up the above inequality from $1$ to $m$, we obtain~\eqref{sta2}.
\end{proof}

We make the following two assumptions, which are reasonable because the homogenization solution~\eqref{eq:u} is essentially smooth.
\begin{assumption}
The elliptic boundary value problem
\[
	\left\{
	\begin{aligned}
		-\nabla\cdot(A\nabla\phi_g)&=g, && \textnormal{in} ~D,\\
		\phi_g&=0, && \textnormal{on} ~\partial D,
	\end{aligned}
	\right. 
\]
and the solution of the dual problem
\begin{equation}
	\label{phig}
	\left\{
	\begin{aligned}
		-\nabla\cdot(A^{\top}\nabla\phi_g)&=g, && \textnormal{in} ~D,\\
		\phi_g&=0, && \textnormal{on} ~\partial D,
	\end{aligned}
	\right.
\end{equation}	
satisfies the regularity estimate
\[
\|\phi_g\|_{H^2(D)}\leq C\|g\|_{L^2(D)}.
\]
\end{assumption}

\begin{assumption}
$u_0 \in H^2(D)\cap H_0^1(D)$ and $u\in H^1(0,T;H^2(D)\cap H_0^1(D))\cap H^2(0,T;L^2(D))$.
\end{assumption}

\begin{remark}
Assumption I holds when $A \in C^{0,1}(D)$. For
Assumption II, a sufficient condition to ensure $u\in H^1(0,T;H^2(D))\cap H^2(0,T;L^2(D))$ is $A \in C^{1,1}(D), u_0 \in H^3(D)$ and $f \in L^2(0, T; H^2(D)) \cap H^1(0, T; L^2(D)).$ 
See, e.g., \cite[\S~6.3, Theorem 4 and \S~7.1, Theorem 6]{evans_partial_2010}. These conditions are sufficient to ensure assumptions, but they are not necessary for the certain solution $u$ under consideration.
\end{remark}

One useful tool is the Galerkin projection operator $R_H:~H_0^1(D)\rightarrow X_{h,H}$ defined by
\begin{equation}\label{eq:glocalProj}
(b^\varepsilon_H \nabla R_H v,\nabla V)_D=(A\nabla v, \nabla V)_D,\qquad\text{for all}\quad V\in X_{h,H}.
\end{equation}

To derive the error bound to the Galerkin projection $R_H$, we recall the following auxiliary result.
\begin{lemma}[{\cite[Lemma 3.1]{huang_concurrent_2018}}]\label{lema:est}
For any $v\in H^1(D)$ with $D\subset\mathbb R^n$ a Lipschitz domain, and for any subset $K\subset D$, then
\begin{equation}
\nm{v}{L^2(K)}\le C\eta(K)\nm{v}{H^1(D)},
\end{equation}
where 
\[
\eta(K)= \begin{cases}
\abs{K}^{1 / 2}\abs{\log | K|}^{1 / 2}, & n=2, \\ 
\abs{K}^{1 /n}, & n\geq 3 ,\end{cases}
\]	
with $\abs{K}$ the measure of $K$, and $C$ is independent of $\abs{K}$.
\end{lemma}

This estimate is the special case of~\cite[Lemma 3.1]{huang_concurrent_2018} with $s=1$ provided that $u$ satisfies Assumption II. If $u$ enjoys higher regularity, e.g., $u\in H^1(0,T;W^{1,\infty}(D))$, then $\eta(K)=|K|^{1/2}$ in any spatial dimension $n$. This rate is sharp, as demonstrated by the example in \cite[Appendix A]{huang_concurrent_2018}. Further discussion on the factor $\eta(K)$ may be found in~\cite{huang_concurrent_2018} and we do not elaborate on this for brevity.

The error estimate for the above Galerkin projection may be found in
\begin{lemma}\label{lemma2}
Under Assumption I, and if $v \in H_0^1(D) \cap H^2(D)$, then
\begin{equation}\label{eq:lema2H1}
	\left.\|\nabla\left(v-R_H v\right)\|_{L^2(D)}\right. \leq C\left(H+\eta(K)+e(\mathrm{HMM})\right)\|v\|_{H^2(D)},
\end{equation}
and
\begin{equation}\label{eq:lema2L2}
	\|v-R_H v\|_{L^2(D)} \leq C\left(H^2+\eta(K)^2
	+e(\mathrm{HMM})\right)\|v\|_{H^2(D)},
\end{equation}	
where
\[
e(\mathrm{HMM})=\max _{x\in D\backslash K}|\left(A-A_H\right)\left(x\right)|.
\]
	
Particularly, if the initial value $U_0$ is the solution of~\eqref{eq:U0}, then
\begin{equation}\label{eq:U0H1}
\nm{\nabla(U_0-R_Hu_0)}{L^2(D)}\le C\left(\eta(K)+e(\mathrm{HMM})\right)\nm{u_0}{H^2(D)},
\end{equation}
and
\begin{equation}\label{eq:U0L2}
\nm{U_0-R_Hu_0}{L^2(D)}\le C\left(H^2+\eta(K)^2+e(\mathrm{HMM})\right)\nm{u_0}{H^2(D)}.
\end{equation}
\end{lemma}

\begin{proof}
Proceeding along the same line that leads to~\cite[Theorem 3.1]{huang_concurrent_2018}, we obtain
\begin{equation}\label{eq:lema2L20}\begin{aligned}
\|v-R_H v\|_{L^2(D)}&\leq C\left( \inf_{V\in X_{h,H}}\|\nabla(v-V)\|_{L^2(D)}+\| \nabla v \|_{L^2(K)}\right)\\
		&\qquad\times \sup_{g\in L^2(D)} \frac{1}{\|g\|_{L^2(D)}}\left( \inf_{V\in X_{h,H}}\|\nabla(\phi_g-V)\|_{L^2(D)}+\| \nabla \phi_g\|_{L^2(K)}\right)\\
		&\qquad\qquad+C e(\mathrm{HMM})\|\nabla v\|_{L^2(D)},
\end{aligned}
\end{equation}
and
\begin{equation}\label{eq:lema2H10}
		\|\nabla (v-R_H v)\|_{L^2(D)} \leq C \Big( \inf_{V\in X_{h,H}}\|\nabla(v-V)\|_{L^2(D)}+\| \nabla v \|_{L^2(K)}+e(\mathrm{HMM})\|\nabla v\|_{L^2(D)}\Big),
\end{equation}
where $\phi_g$ is the solution of \eqref{phig}.
	
As to the first term, 
\[
	\inf_{V\in X_{h,H}} \| \nabla (v-V) \|_{L^2(D)}\leq \|\nabla(v-\Pi v)\|_{L^2(D)}\leq CH\|v\|_{H^2(D)},
\]
where $\Pi:H_0^1(D)\to X_{h,H}$ is the linear Lagrange interpolation, and also
\[
	\inf_{V\in X_{h,H}} \| \nabla (\phi_g-V) \|_{L^2(D)}\leq CH\|\phi_g\|_{H^2(D)}\leq CH\|g\|_{L^2(D)},
\]
because of Assumption I.
	
As to the second term, using Lemma~\ref{lema:est}, we obtain
\[
\| \nabla v \|_{L^2(K)}\leq C \eta (K) \| v \|_{H^2(D)},
\]
and
\[
	\| \nabla \phi_g \|_{L^2(K)} \leq C \eta (K) \| \phi_g \|_{H^2(D)}\leq C \eta(K) \| g \|_{L^2(D)}.
\]
Substituting all the above inequalities into~\eqref{eq:lema2L20} and~\eqref{eq:lema2H10}, we obtain~\eqref{eq:lema2L2} and~\eqref{eq:lema2H1}.

Next,if $U_0$ is the solution of~\eqref{eq:U0}, then
\[
(b_H^\varepsilon\nabla(U_0-R_Hu_0),\nabla V)_D=((b_H^\varepsilon-A)\nabla u_0,\nabla V)_D\qquad\text{for all}\quad V\in X_{h,H}.
\]
Hence,
\begin{align*}
((b_H^\varepsilon-A)\nabla u_0,\nabla V)_D&\le C\Bigl(\nm{\nabla u_0}{L^2(K)}+e(\mathrm{HMM})\nm{\nabla u_0}{L^2(D)}\Bigr)\nm{\nabla V}{L^2(D)}\\
&\le C(\eta(K)+e(\mathrm{HMM}))\nm{u_0}{H^2(D)}\nm{\nabla V}{L^2(D)}.
\end{align*}
Choosing $V=U_0-R_Hu_0$, we obtain ~\eqref{eq:U0H1}.

To obtain the $L^2$-estimate, we use the Aubin-Nitsche’s dual argument~\cite{Nitsche:1968kriterium}. For any $g\in L^2(D)$, let $\phi_g$ be the solution of~\eqref{phig}, we have
\begin{equation}\label{eq:U0L20}\begin{aligned}
(g,U_0-R_Hu_0)_D&=(A\nabla(U_0-R_Hu_0),\nabla\phi_g)_D\\
&=(A\nabla(U_0-R_Hu_0),\nabla(\phi_g-\Pi\phi_g))_D+\Lr{(A-b_H^\varepsilon)\nabla(U_0-R_Hu_0),\nabla \Pi\phi_g}_D\\
&\qquad+(b_H^\varepsilon\nabla(U_0-R_Hu_0),\nabla \Pi\phi_g)_D.
\end{aligned}\end{equation}
The first term in~\eqref{eq:U0L20} is bounded by
\begin{align*}
(A\nabla(U_0-R_Hu_0),\nabla(\phi_g-\Pi\phi_g))_D&\le CH\Lr{\eta(K)+e(\mathrm{HMM})}\nm{u_0}{H^2(D)}\nm{g}{L^2(D)}\\
&\le C\Bigl(H^2+\eta(K)^2+e(\mathrm{HMM})^2\Bigr)\nm{u_0}{H^2(D)}\nm{g}{L^2(D)}.
\end{align*}

As to the second term in~\eqref{eq:U0L20}, we have
\begin{align*}
((A-b_H^\varepsilon)\nabla(U_0-R_Hu_0),\nabla \Pi\phi_g)_D&\le C\Bigl(\nm{\nabla(\phi_g-\Pi\phi_g)}{L^2(D)}+\nm{\nabla\phi_g}{L^2(K)}+e(\mathrm{HMM})\nm{\nabla\phi_g}{L^2(D)}\Bigr)\\
&\qquad\times\nm{\nabla(U_0-R_Hu_0)}{L^2(D)}\\
&\le C\Bigl(H^2+\eta(K)^2+e(\mathrm{HMM})^2\Bigr)\nm{u_0}{H^2(D)}\nm{g}{L^2(D)}.
\end{align*}

Finally, the last term in~\eqref{eq:U0L20} may be bounded as
\begin{align*}
(b_H^\varepsilon\nabla(U_0-R_Hu_0),\nabla \Pi\phi_g)_D&=((b_H^\varepsilon-A)\nabla u_0,\nabla\Pi\phi_g)_D\\
&\le C\Bigl(\nm{\nabla u_0}{L^2(K)}\nm{\nabla\Pi\phi_g}{L^2(K)}+e(\mathrm{HMM})\nm{\nabla u_0}{L^2(D)}\nm{\nabla\Pi\phi_g}{L^2(D)}\Bigr)\\
&\le C\Bigl(H^2+\eta(K)^2+e(\mathrm{HMM})\Bigr)\nm{u_0}{H^2(D)}\nm{g}{L^2(D)}.
\end{align*}

Substituting the above three inequalities into~\eqref{eq:U0L20}, we obtain~\eqref{eq:U0L2}.
\end{proof}

Based on the above lemma, we derive the error bound for the
approximation of the homogenized solution.
\begin{theorem}\label{thm:global}
Under the assumptions I and II, there holds
\begin{equation}\label{thm11}
\begin{aligned}
\|u_m-U_m\|_{L^2(D)}&\leq C\left(H^2+\eta(K)^2+e(\mathrm{HMM})\right)\left(\|u_0 \|_{H^2(D)}+\|\partial_t u\|_{L^2(0,T;H^2(D))}\right)\\
&\qquad+C \Delta t \|\partial_t^2 u\|_{L^2(0, T;L^2(D))},
\end{aligned}
\end{equation}
and
\begin{equation}
		\label{thm12}
	\begin{aligned}
		\|\nabla(u_m-U_m)\|_{L^2(D)}&\leq C\bigl(H+\eta(K)+e(\mathrm{HMM})\bigr)\left(\|u_0 \|_{H^2(D)}+\|\partial_t u\|_{L^2(0,T;H^2(D))}\right)\\
		&\qquad+C \Delta t \|\partial_t^2 u\|_{L^2(0, T;L^2(D))},
    \end{aligned}
\end{equation}
and
\begin{equation}\label{thm13}
\begin{aligned}
\|\partial_t u-\delta U\|_{\ell^2(0,T;L^2(D))} &\leq C\left(H^2+\eta(K)+e(\mathrm{HMM})\right)
		\left(\|u_0 \|_{H^2(D)}+\|\partial_t u\|_{L^2(0, T; H^2(D))}\right)\\
&\qquad+C \Delta t\|\partial_t^2 u\|_{L^2(0,T;L^2(D))}.
	\end{aligned}
	\end{equation}
\end{theorem}	

\begin{proof}
We write
	$$u_m-U_m=u_m-R_H u_m+R_H u_m-U_m.$$
	By Lemma \ref{lemma2}, we have
	\begin{equation}
	\label{thm14}
	\|u_m-R_H u_m\|_{L^2(D)} \leq C\left(H^2+\eta(K)^2+e(\mathrm{HMM})\right)\|u_m\|_{H^2(D)},
	\end{equation}
	and 
	\begin{equation}
	\label{thm15}
	\left.\|\nabla\left(u_m-R_H u_m\right)\|_{L^2(D)}\right. \leq C\big(H+\eta(K)+e(\mathrm{HMM})\big)\|u_m\|_{H^2(D)},
	\end{equation}
	where
	\[
		\|u_m\|_{H^2(D)} \leq\|u_0\|_{H^2(D)}
		+\int_0^T\|\partial_t u(\cdot, t)\|_{H^2(D)}\mathrm{d}\,t\leq\|u_0\|_{H^2(D)}+\sqrt{T}\|\partial_t u\|_{L^2(0, T; H^2(D))}.
	\]
	
On the one hand, we write
\begin{align*}
(\delta\left(U_k-R_H u_k), V\right)_D+\left(b^\varepsilon_H \nabla(U_k-R_H u_k), \nabla V\right)_D&=\left(f_k, V\right)_D-\left(\delta R_H u_k, V\right)_D-(A \nabla u, \nabla V)_D \\
	&=\left(\partial_t u_k-\delta R_H u_k, V\right)_D=\left(\partial_t u_k-R_H \delta u_k, V\right)_D,
	\end{align*}
where $R_H$ and $\delta$ are changeable because $A$ and $b_h^\varepsilon$ are time-independent. Using the stability
estimates~\eqref{sta1} and~\eqref{sta2}, we obtain 
\begin{equation}\label{sta3}
	\|U_m-R_{H} u_m\|_{L^2(D)} \leq \|U_0-R_H u_0\|_{L^2(D)}+C\|\partial_t u-R_H \delta u\|_{\ell^2(0, T; H^{-1}(D))},
\end{equation}
and
\begin{equation}\label{sta4}
\begin{aligned}
		\|\delta (U-R_H u)\|_{\ell^2(0,T;L^2(D))}+\|\nabla\left(U_m-R_H u_m\right)\|_{L^2(D)}&\leq\|\nabla\left(U_0-R_H u_0\right)\|_{L^2(D)}\\
		&\qquad+C\|\partial_t u-R_H \delta u\|_{\ell^2(0, T; L^2(D)) }.
\end{aligned}
	\end{equation}
Estimates for the initial value term have been shown in Lemma~\ref{lemma2}. As to the source term, we write
\begin{equation}\label{sourceterm}
\partial_t u_k-R_H \delta u_k=(\partial_t u_k-\delta u_k)
+(\delta u_k-R_H \delta u_k) \text {. }
\end{equation}
	
For the first term in \eqref{sourceterm}, we have
\begin{align*}
\left(\partial_t u_k-\delta u_k\right)(x) & =\frac{1}{\Delta t} \int_{t_{k-1}}^{t_k}\left(\partial_t u\left(x, t_k\right)-\partial_t u(x, t)\right)\mathrm{d}\,t=\frac{1}{\Delta t} \int_{t_{k-1}}^{t_k}\left(t-t_{k-1}\right) \partial_t^2 u(x, t)\mathrm{d}\,t  \\
& \leq \sqrt{\Delta t}\|\partial_t^2 u(x, \cdot)\|_{L^2\left(t_{k-1}, t_k\right)}.
\end{align*}
Therefore,
\[
		\|\partial_t u-\delta u\|_{\ell^2(0, T;L^2(D))}^2\leq \Delta t^2 \sum_{k=1}^m\|\partial_t^2 u\|_{L^2(t_{k-1}, t_k ;L^2(D))}^2=\Delta t^2\|\partial_t^2 u\|_{L^2(0, T ;L^2(D))}^2 .
	\]
As to the second term in ~\eqref{sourceterm}, we obtain
\begin{align*}
		 \left(\delta u_k-R_H \delta u_k\right)(x)&=\frac{I-R_{H}}{\Delta t} \int_{t_{k-1}}^{t_k} \partial_t u(x, t) d t=\frac{1}{\Delta t} \int_{t_{k-1}}^{t_k}\left(\partial_t u-R_{H} \partial_t u\right)(x, t) d t \\
		& \leq \frac{1}{\sqrt{\Delta t}}\|\left(\partial_t u-R_H \partial_t u\right)(x, \cdot)\|_{L^2(t_{k-1}, t_k)}.
	\end{align*}
Therefore, we have
\[
\|\delta u-R_H \delta u\|_{\ell^2(0,T;L^2(D))}^2
\leq \sum_{k=1}^m\|\partial_t u-R_H \partial_t u\|_{L^2(t_{k-1}, t_k,L^2(D))}^2=\|\partial_t u-R_H \partial_t u\|^2_{L^2(0,T;L^2(D))}.
\]
Combining above two inequalities with the triangle inequality, we get
\[
	\begin{aligned}
		 \|\partial_t u-R_H \delta u\|_{\ell^2(0, T; H^{-1}(D))}&\leq \|\partial_t u-R_H \delta u\|_{\ell^2(0, T ; L^2(D))}\\
		&\leq \|\partial_t u-\delta u\|_{\ell^2(0,T;L^2(D))}+\|\delta u-R_H\delta u\| _{\ell^2(0, T;L^2(D))} \\
		&\leq \Delta t\|\partial_t^2 u\|_{L^2(0,T;L^2(D))}+\|\partial_t u-R_H \partial_t u\|_{L^2(0,T;L^2(D))} \\
		&\leq \Delta t\|\partial_t^2 u\|_{L^2(0,T;L^2(D))}+C\left(H^2+\eta(K)^2+e(\mathrm{HMM})\right)\|\partial_t u\|_{L^2(0, T; H^2(D))}. \\
	\end{aligned}
\]

Substituting the above inequalities into~\eqref{sta3} 
and~\eqref{sta4}, we obtain
\begin{equation}
\begin{aligned}
\|U_m-R_H u_m\|_{L^2(D)}&\leq C\Bigl(H^2+\eta(K)^2+e(\mathrm{HMM})\Bigr)
\left(\|u_0\|_{H^2(D)}+
 \|\partial_t u\|_{L^2(0, T; H^2(D))}\right)\\
&\qquad+\Delta t \|\partial_t^2 u\|_{L^2(0,T;L^2(D))},
\end{aligned}
\end{equation}
and
\begin{equation}
\begin{aligned}
	\|\delta(U-R_Hu)\|_{\ell (0,T;L^2(D))}&+\|
	\nabla (U_m-R_H u_m)\|_{L^2(D)}
	\leq \Delta t \|\partial_t^2 u\|_{L^2(0,T;L^2(D))}\\
	&\quad+C\Bigl(H^2+\eta(K)+e(\mathrm{HMM})\Bigr)
	\left(\|u_0\|_{H^2(D)}+
	\|\partial_t u\|_{L^2(0, T; H^2(D))}\right).
	\end{aligned}
\end{equation}

Combining~\eqref{thm14} and~\eqref{thm15} with the triangle inequality, we obtain estimates of $u_m-U_m$ in  \eqref{thm11} and~\eqref{thm12}. Moreover,
$$
\|\partial_t u-\delta U\|_{\ell(0,T;L^2(D))}
\leq \|\partial_t u-R_H\delta u\|_{\ell(0,T;L^2(D))}
+ \|\delta (U-R_H u)\|_{\ell(0,T;L^2(D))},
$$
which gives~\eqref{thm13}.
\end{proof}

\begin{remark}
In~\cite[Theorem 1.1]{mingAnalysisHeterogeneousMultiscale2007}, for the general numerical homogenization solution $U$ calculated by HMM, the term $\|\nabla(u_m-U_m)\|_{L^2(D)}$ is bounded by
$\mathcal{O}(\Delta t+H+e(\mathrm{HMM})\Delta t^{-1/2})$ . 
The above estimate eliminates the factor $\Delta t^{-1/2}$ when the diffusion coefficient is time-independent, which indicates that $\Delta t$ can be arbitrarily small.
\end{remark}
\subsection{Accuracy for retrieving the local microscopic information}

In order to prove the localized energy error estimate, we introduce some notation and results  in \cite{demlow_local_2011}. For a subdomain $B \subset D$, define $H_{<}^1(B):=\set{u \in H^1(D)|u|_{D \backslash B}=0}$. Let $G_1$ and $G$ be subsets of $K$ with $G_1 \subset G$ and $\operatorname{dist}\left(G_1, \partial G \backslash \partial D\right)=d>0$. The following assumptions are assumed to hold:

A1: Local interpolant. There exists a local interpolant such that for any $u \in H_{<}^1\left(G_1\right) \cap$ $C\left(G_1\right), I u \in X_{h,H}\cap H_{<}^1(G)$.

A2: Inverse properties. For each $\chi \in X_{h,H}$ and $\tau \in \mathcal{T}_{h,H}\cap K, 1 \leq p \leq q \leq \infty$, and $0 \leq \nu \leq s \leq r$,
$$
\|\chi\|_{W^{s, q}(\tau)} \leq C h_\tau^{\nu-s+n / p-n / q}\|\chi\|_{W^{\nu, p}(\tau)} .
$$

A3. Superapproximation. Let $\omega \in C^{\infty}(K) \cap H_{<}^1\left(G_1\right)$ with $|\omega|_{W^{j, \infty}{ }{(K)}} \leq C d^{-j}$ for integers $0 \leq j \leq r$, for each $\chi \in X_{h,H} \cap H_{<}^1(G)$ and for each $\tau \in \mathcal{T}_{h,H}\cap K$ satisfying $h_\tau \leq d$,
\begin{equation}
\label{superapp}
\|\omega^2 \chi-I(\omega^2 \chi)\|_{H^1(\tau)} \leq C\left(\frac{h_\tau}{d}\|\nabla(\omega \chi)\|_{L^2(\tau)}+\frac{h_\tau}{d^2}\|\chi\|_{L^2(\tau)}\right),
\end{equation}
where the interpolant $I$ is defined in A1. The assumptions A1,A2 and A3 are satisfied by standard Lagrange finite element defined
on shape-regular grids \cite{demlow_local_2011}.

Set $K_0 \subset K_1\subset K$, with $\text{dist}(K_1,\partial K\backslash \partial D)=\text{dist}(K_0,\partial K_1\backslash \partial D)=cd$ for a universal constant $c$. The following lemma is a generalization of~\cite[Lemma 3.3]{demlow_local_2011} by including the effect of the source term. Similar results have been proven in~\cite{nitsche_interior_1974,jinchao2001local} without clarifying the dependence on the distance $d$. 
\begin{lemma}\label{lemma3}
For $W\in X_{h,H}$ satisfies
\[
(b^\varepsilon_H\nabla W, \nabla V)_D=(f,V)_D \qquad  \text{for all\quad}  V\in X_{h,H}\cap H_<^1(K_1),
\]
then
\begin{equation}\label{eq:localest}
	\|\nabla W\|_{L^2\left(K_0\right)} \leq C\Lr{d^{-1}\|W\|_{L^2\left(K_1\right)}+\|f\|_{H^{-1}\left(K_1\right)}}.
\end{equation}
\end{lemma}

\begin{proof}
We set $\omega \in C^\infty(K_1)\cap H_<^1(K_1)$ a cut-off function with $\omega \equiv 1$ on $K_0$ and $|\nabla^j \omega|\leq Cd^{-j},~j=0,1$. A direct manipulation gives
\begin{align*}
(b^\varepsilon_H \nabla(\omega W), \nabla(\omega W))_D&=(b^\varepsilon_H \nabla W, \nabla(\omega^2 W))_D
+(b^\varepsilon_HW\nabla\omega, W\nabla\omega)_D\\
&\qquad+(b^\varepsilon_HW\nabla\omega,\nabla(\omega W))_D-(b^\varepsilon_H\nabla(\omega W),W\nabla\omega)_D\\
&\leq (b^\varepsilon_H \nabla W, \nabla(\omega^2 W))_D+\frac{\lambda}{4}\|\nabla(\omega W)\|^2_{L^2\left(K_1\right)}+\frac{C}{d^2 }\|W\|^2_{L^2\left(K_1\right)}.
\end{align*}

The first term may be reshape into
\[
(b^\varepsilon_H \nabla W, \nabla(\omega^2 W))_D=(b^\varepsilon_H \nabla W, \nabla(\omega^2 W-I(\omega^2 W))
)_D+(f,I(\omega^2 W))_D.
\]
Using the superapproximation~ \eqref{superapp} and the inverse estimate, we get
\begin{align*}
(b^\varepsilon_H \nabla W, \nabla(\omega^2 W))_D
& \leq C \|\nabla W\|_{L^2(K_1)} \| \nabla(\omega^2 W)-\nabla I(\omega^2 W)\|_{L^2(K_1)}+\|f\|_{H^{-1}\left(K_1\right)}\| I(\omega^2 W) \|_{H^1\left(K_1\right)}\\
&\le C\left(d^{-1}\|W\|_{L^2\left(K_1\right)}+\|f\|_{H^{-1}\left(K_1\right)}\right)\left(\|\nabla(\omega W)\|_{L^2\left(K_1\right)}+d^{-1}\|W\|_{L^2\left(K_1\right)}\right),
\end{align*}
where we have assumed $h\leq d$.

Combining the above inequalities, there holds
\begin{align*}
\lambda\nm{\nabla(\omega W)}{L^2(K_1)}^2&\le(b^\varepsilon_H\nabla(\omega W),\nabla(\omega W))_D\\
&\le\frac{\lambda}{2}\|\nabla(\omega W)\|^2_{L^2\left(K_1\right)}+C\left(d^{-2}\|W\|_{L^2\left(K_1\right)}^2+\|f\|_{H^{-1}(K_1)}^2\right).
\end{align*}
This gives~\eqref{eq:localest}.
\end{proof}

We are ready to give the localized energy error estimate. Firstly, we define a local Galerkin projection operator $R_h:H_<^1(K) \rightarrow X_{h,H}\cap H_<^1(K)$ such that
$$(b^\varepsilon_H \nabla R_h v,\nabla V)_D=(a^\varepsilon \nabla v,\nabla V)_D \qquad \text{for all }V\in X_{h,H}\cap H_<^1(K).$$

The interior estimate relies on the Meyers regularity for linear parabolic equations, which is given in \cite[Chapter 2, Theorem 2.5]{bensoussan_asymptotic_2011}.

\begin{lemma}
	Let $D$ be a bounded domain, $f\in L^2(0,T;H^{-1}(D))$, $u_0\in H_0^1(D)$ and $u\in L^2(0,T;H_0^1(D))$ be the solution of
	\[\left\{\begin{aligned}
	    \partial_t u-\nabla\cdot(A\nabla u)&=f,&&\text{in\;} D_T,\\
	    u|_{t=0}&=u_0,&&\text{in\;} D,
	\end{aligned}\right.\]
	where $A\in [L^\infty(D_T)]^{d\times d}$ is uniformly elliptic with corresponding parameters $\lambda$ and $\Lambda$. There exists $p>2$, depending on $\lambda$, $\Lambda$ and $D$, such that if
	$u_0\in W_0^{1,p}(D)$ and $f\in L^p(0,T;W^{-1,p}(D))$, then $u\in L^p(0,T;W_0^{1,p}(D))$. Furthermore,
	$$\Vert u \Vert_{L^p(0,T;W^{1,p}(D))}\leq C\Bigl(\Vert f \Vert_{L^p(0,T;W^{-1,p}(D))}+T^{1/p}\Vert u_0 \Vert_{W{1,p}(D)}\Bigr),$$
	where the constant $C > 0$ depends on $\lambda,\Lambda$ and $D$.
\end{lemma}

\begin{theorem}\label{thm:localest1}
If $u^\varepsilon\in \ell^2(0,T;W^{1,p}(K))$ for certain $p>2$, then
\begin{equation}\label{eq:3.6}
\begin{aligned}
\|\nabla (u^\varepsilon-U)\|_{\ell^2(0,T;L^2(K_0))}& \leq  C \bigg( \inf_{V\in [X_{h,H}\cap H_<^1(K)]^m} \| \nabla(\omega u^\varepsilon-V) \|_{\ell^2(0,T;L^2(K_0))} \\
	&\qquad	+ |K\backslash K_0|^{1/2-1/p}\left(d^{-1}\|u^\varepsilon\|_{\ell^2(0,T;L^p(K))} + \| \nabla u^\varepsilon \|_{\ell^2(0,T;L^p(K))}  \right)\\
	&\qquad	+ d^{-1} \|u^\varepsilon-U\|_{\ell^2(0,T;L^2(K))}+\|\partial_t u^\varepsilon-\delta U\|_{\ell^2(0,T;H^{-1}(K))}\bigg),
\end{aligned}
\end{equation}
where $\omega\in C^\infty(K)\cap H_<^1(K)$ with $\omega=1$ on $K_1$ and $|\nabla^j \omega|\leq Cd^{-j},~j=0,1$.
\end{theorem}

\begin{proof}
For arbitrary $V\in X_{h,H}\cap H_{<}^1(K)$,
\[
	\begin{aligned}
		\left(b^\varepsilon_H \nabla\left(R_h v-V\right), \nabla\left(R_h v-V\right)\right)_D&=\left(a^\varepsilon \nabla v-b^\varepsilon_H \nabla V, \nabla\left(R_h v-V\right)\right)_D\\
		&=\left(\left(a^\varepsilon-b^\varepsilon_H\right) \nabla v, \nabla\left(R_h v-V\right)\right)_D+\left(b^\varepsilon_H \nabla(v-V), \nabla\left(R_h v-V\right)\right)_D.
	\end{aligned}
\]
This gives 
\[
	\|\nabla\left(R_h v-V\right)\|_{L^2(K)} \leq C\left(\|\nabla v\|_{L^2\left(K \backslash K_0\right)}+\|\nabla(v-V)\|_{L^2(K)}\right).
\]
By the triangle inequality and the arbitrariness of $V$, it holds that
\[
	\|\nabla\left(v-R_h v\right)\|_{L^2(K)} \leq C \left( \inf_{V\in X_{h,H}\cap H_<^1(K)}\|\nabla(v-V)\|_{L^2(K)}+\|\nabla v\|_{L^2(K\backslash K_0)}\right).
\]
	Set $v=\omega u_k^\varepsilon$, we obtain
\[
\begin{aligned}
\|\nabla (I-R_h)(\omega u_k^\varepsilon)\|_{L^2(K)} &\leq C \left( \inf_{V\in X_{h,H}\cap H_<^1(K)}\|\nabla(\omega u_k^\varepsilon-V)\|_{L^2(K)}+\|\nabla (\omega u_k^\varepsilon)\|_{L^2(K\backslash K_0)}\right)\\
		&\leq C \bigg( \inf_{V\in X_{h,H}\cap H_<^1(K)}\|\nabla(\omega u_k^\varepsilon-V)\|_{L^2(K)}\\
		&\qquad +|K\backslash K_0|^{1/2-1/p}(d^{-1}\|u_k^\varepsilon\|_{L^p(K)} + \| \nabla u_k^\varepsilon \|_{L^p(K)} )\bigg).\\
\end{aligned}
\]
	
Note that for all $V\in X_{h,H}\cap H_<^1(K_1)$, we have
\[
\left(b^\varepsilon_H \nabla\left(U_k-R_h\left(\omega u_k^{\varepsilon}\right)\right), \nabla V\right)_D=\left(f_k-\delta U_k, V\right)_D-\left(a^{\varepsilon} \nabla u_k^\varepsilon, \nabla V\right)_D=\left(\partial_t u_k^\varepsilon-\delta U_k, V\right)_D.
\]

By Lemma~\ref{lemma3}, we obtain
	\begin{align*}
		 \| \nabla&\left(U_k-R_h\left(\omega u_k^\varepsilon\right)\right)  \|_{L^2\left(K_0\right)}\leq C\left(d^{-1}\|U_k-R_h\left(\omega u_k^\varepsilon\right)\|_{L^2\left(K_1\right)}+\|\partial_t u_k^\varepsilon-\delta U_k\|_{H^{-1}\left(K_1\right)}\right)\\
		 &\leq C\left( d^{-1}\|\omega u_k^\varepsilon-R_h\left(\omega u_k^\varepsilon\right)\|_{L^2(K)}+d^{-1}\|u_k^\varepsilon-U_k\|_{L^2\left(K_1\right)}+\|\partial_t u_k^\varepsilon-\delta U_k\|_{H^{-1}\left(K_1\right)}\right)\\
		&\leq C\left(\|\nabla\left(I-R_h\right)\left(\omega u_k^\varepsilon\right)\|_{L^2(K)}+d^{-1}\|u_k^\varepsilon-U_k\|_{L^2(K)}+\|\partial_t u_k^\varepsilon-\delta U_k\|_{H^{-1}(K)}\right),\\
	\end{align*}
where we have used the Poincar\'e inequality in the last line. By the triangle inequality,
	\begin{align*}
	\|\nabla\left(u_k^\varepsilon-U_k\right)\|_{L^2\left(K_0\right)} &\leq\|\nabla\left(I-R_h\right)\left(\omega u_k^\varepsilon\right)\|_{L^2(K)}+\|\nabla\left(U_k-R_h\left(\omega u_k^\varepsilon\right)\right)\|_{L^2\left(K_0\right)}\\
	& \leq  C \bigg( \inf_{V\in X_{h,H}\cap H_<^1(K)} \| \nabla(\omega u^\varepsilon_k-V) \|_{L^2(K_0)} \notag \\
		&\qquad+ |K\backslash K_0|^{1/2-1/p}\left(d^{-1}\|u^\varepsilon_k\|_{L^p(K)} + \| \nabla u^\varepsilon_k \|_{L^p(K)}  \right)\\
		&\qquad+ d^{-1} \|u^\varepsilon_k-U_k\|_{L^2(K)}+\|\partial_t u^\varepsilon_k-\delta U_k\|_{H^{-1}(K)}\bigg).
	\end{align*}
	Sum up $k$ from 1 to $m$, we obtain \eqref{eq:3.6}.
\end{proof}

Under the same condition of Theorem~\ref{thm:localest1}, we may bound the last two terms in the right-hand side of~\eqref{eq:3.6}.
\begin{remark}\label{rmk:local}
Using H\"older inequality and the estimate~\eqref{thm11}, we obtain
\[
	\begin{aligned}
		\|u^{\varepsilon}-U\|_{\ell^2(0, T ; L^2(K))} &\leq \|u^{\varepsilon}-u\|_{\ell^2(0, T; L^2\left(K\right))}+\|u-U\|_{\ell^2(0, T ; L^2(D))}\\
		&\le C\abs{K}^{1/2-1/p}\Lr{\nm{u^{\varepsilon}}{\ell^2(0,T;L^p(K))}
		+\nm{u}{\ell^2(0,T;L^p(K))}}\\
		&\qquad+C(u)\left(\Delta t+H^2+\eta(K)^2+e(\mathrm{HMM})\right).
	\end{aligned}
\]

Note that for any $0<t\le T$ and for $p>2$, denoting by $p'$ the conjugate index of $p$, we get 
\begin{align*}
\nm{\partial_t(u^{\varepsilon}-u)(\cdot,t)}{H^{-1}(K)}
&=\sup_{\varphi\in H_0^1(K)}\dfrac{\abs{(\partial_t(u^{\varepsilon}-u)(\cdot,t),\varphi)}}{\nm{\na\varphi}{L^2(K)}}=\sup_{\varphi\in H_0^1(K)}\dfrac{\abs{(a^\varepsilon\na u^{\varepsilon}-A\na u)(\cdot,t),\na\varphi)}}{\nm{\na\varphi}{L^2(K)}}\\
 &\le\Lambda\sup_{\varphi\in H_0^1(K)}\dfrac{\Lr{\nm{\na u^{\varepsilon}(\cdot,t)}{L^p(K)}+\nm{\na u(\cdot,t)}{L^p(K)}}\nm{\na\varphi}{L^{p'}(K)}}{\nm{\na\varphi}{L^2(K)}}\\
 &\le\varLambda\abs{K}^{1/p'-1/2}
 \Lr{\nm{\na u^{\varepsilon}(\cdot,t)}{L^p(K)}+\nm{\na u(\cdot,t)}{L^p(K)}}\\
 &=\varLambda\abs{K}^{1/2-1/p}
 \Lr{\nm{\na u^{\varepsilon}(\cdot,t)}{L^p(K)}+\nm{\na u(\cdot,t)}{L^p(K)}}.
\end{align*}

Using~\eqref{thm13} and note that $u^\varepsilon\in \ell^2(0,T;W_0^{1,p}(K))$, we get
\[
\begin{aligned}
		\|\partial_t u^\varepsilon-\delta U\|_{\ell^2(0,T;H^{-1}\left(K\right))} &\leq\|\partial_t u^\varepsilon-\partial_t u\|_{\ell^2(0,T; H^{-1}\left(K\right))}+\|\partial_t u-\delta U\|_{\ell^2(0,T; H^{-1}(K))}\\
		&\le C(u^\varepsilon,u)\abs{K}^{1/2-1/p}+C(u)\Lr{\Delta t+H^2+\eta(K)+e(\mathrm{HMM})}.
\end{aligned}
\]
\end{remark}
Theorem~\ref{thm:localest1} relies on Meyers' regularity estimate, which introduce an additional factor $\abs{K\backslash K_0}^{1/2-1/p}$. This factor cannot be ignored, as $u^\varepsilon$ may exhibit severe oscillation. The next theorem shows that choosing $\rho=\chi_K$ removes this factor, thereby eliminating the pollution originating from outside $K_0$. This observation clarifies the poor performance of $\rho=\chi_{K_0}$ observed in the elliptic-case experiments of \cite{huang_concurrent_2018}. In practice, for continuous transition functions, the pollution outside $K_0$ remains acceptable.
\begin{theorem}\label{thm:local}
If $\rho=\chi_K$, then
	\begin{equation}
	\label{thm2}
	\begin{aligned}
		\| \nabla(u^\varepsilon-U)\| _{\ell^2(0,T;L^2(K_0))} &\leq C\left(\inf_{V\in[X_{h,H}]^m} \| \nabla(u^\varepsilon-V)\|_{\ell^2(0,T;L^2(K))}\right. \\
		&\qquad\qquad + d^{-1} \| u^\varepsilon-U\| _{\ell^2(0,T;L^2(K))} + \| \partial_t u^\varepsilon-\delta U \|_{\ell^2(0,T;H^{-1}(K))} \bigg).\\
	\end{aligned}
	\end{equation}
\end{theorem}

\begin{proof}
	Set $\chi_k=\underset{V\in X_{h,H}}{\operatorname{arg\,min}} \|\nabla(u_k^\varepsilon-V)\|_{L^2(K)}$ with $\int_K\chi_k \mathrm{d}x=\int_K u_k^\varepsilon \mathrm{d}x$. Define a cut-off function $\omega\in C^\infty(K)\cap H_<^1(K)$ with $\omega=1$ on $K_1$ and
	$|\nabla^j \omega|\leq Cd^{-j},~j=0,1$.
	$$
	\begin{aligned}
		\|\nabla(u^\varepsilon_k-U_k)\|_{L^2(K_0)}&\leq \|\nabla(\omega u^\varepsilon_k-\omega\chi_k)-\nabla R_h(\omega u_k^\varepsilon-\omega \chi_k)\|_{L^2(K)}\\
		&\qquad +\|\nabla(U_k-\chi_k)-\nabla R_h(\omega u_k^\varepsilon-\omega \chi_k)\|_{L^2(K_0)}.\\
	\end{aligned}
	$$
	By $\int_K (u^\varepsilon_k-\chi_k) \mathrm{d}x=0$ and Poincaré inequality on $K$, we have
\begin{equation}\label{eq:thm3}
\begin{aligned}
		\|\nabla(\omega u^\varepsilon_k-\omega\chi_k)-\nabla R_h(\omega u_k^\varepsilon-\omega \chi_k)\|_{L^2(K)} &\leq
		C\|\nabla (\omega u_k-\omega \chi_k)\|_{L^2(K)}\\
		&\leq C\left( \|\nabla(u_k^\varepsilon-\chi_k)\|_{L^2(K)}+d^{-1}\|u_k^\varepsilon-\chi_k\|_{L^2(K)}  \right)\\
		&\leq C \|\nabla(u_k^\varepsilon-\chi_k)\|_{L^2(K)}\\
		&\leq C \inf_{V\in X_{h,H}}\| \nabla(u_k^\varepsilon-V) \|_{L^2(K)}, 
	\end{aligned}
\end{equation}
where we use the relation
\[\nm{\nabla R_hv}{L^2(K)}\le\frac{\Lambda}{\lambda}\nm{\nabla v}{L^2(K)},\qquad\text{for all}\quad v\in H^1_<(K)\]
in the first step. Noting that $\omega=1$ and $b^\varepsilon_H=a^\varepsilon$ on $K_1$, for all $V\in X_{h,H}\cap H_<^1(K_1)$, we obtain
\[
\begin{aligned}
		&\left( b^\varepsilon_H \nabla \left(U_k-\chi_k-R_h(\omega u_k^\varepsilon-\omega \chi_k)\right),\nabla V \right)_D\\
		&=(f_k-\delta U_k,V)_D-(b^\varepsilon_H\nabla\chi_k,\nabla V)_D-\left(a^{\varepsilon} \nabla\left(\omega u_k^{\varepsilon}-\omega \chi_k\right), \nabla V\right)_D\\
		&=\left(f_k-\delta U_k, V\right)_D-\left(a^\varepsilon \nabla u_k^\varepsilon, \nabla V\right)_D=\left(\partial_t u_k^\varepsilon-\delta U_k, V\right)_D.
	\end{aligned}
\]
By Lemma \ref{lemma3} and invoking the Poincaré inequality again, we obtain
\[
	\begin{aligned}
		&\|\nabla\left(U_k-\chi_k\right)-\nabla R_h\left(\omega u_k^\varepsilon-\omega \chi_k\right)\|_{L^2\left(K_0\right)}\\
		&\leq C\left(d^{-1}\|U_k-\chi_k-R_h\left(\omega u_k^{\varepsilon}-\omega \chi_k\right)\|_{L^2\left(K_1\right)}+\|\partial_t u_k^\varepsilon-\delta U_k\|_{H^{-1}\left(K_1\right)}\right)\\
		&\leq C\left(d^{-1}\|\omega u_k^\varepsilon-\omega \chi_k-R_h\left(\omega u_k^\varepsilon-\omega \chi_k\right)\|_{L^2(K)}+d^{-1}\|u_k^{\varepsilon}-U_k\|_{L^2\left( K_1\right)}+\|\partial_t u_k^\varepsilon-\delta U_k\|_{H^{-1}\left(K_1\right)}\right)\\
		&\leq C\left(\|\nabla(\omega u_k^\varepsilon-\omega \chi_k)-\nabla R_h\left(\omega u_k^{\varepsilon}-\omega \chi_k\right)\|_{L^2(K)}+d^{-1}\|u_k^\varepsilon-U_k\|_{L^2(K_1)}+\|\partial_t u_k^\varepsilon-\delta U_k\|_{H^{-1}\left(K_1\right)}\right).
	\end{aligned}
\]
Substituting~\eqref{eq:thm3} into above inequality again and summing up from $k=1$ to $m$, we obtain~\eqref{thm2}.
\end{proof}

So far, we have not discussed the estimate of $e(\mathrm{HMM})$, the so-called resonance error, which has been studied extensively in~\cite{Ming:2005} and~\cite{mingAnalysisHeterogeneousMultiscale2007}. We also refer to~\cite{MR3451428} for certain more advanced techniques to reduce the resonance error.
%
\section{Numerical examples}~\label{sec:numer}
In this part, we present several numerical examples to validate the accuracy of the proposed method. In all cases, the domain is set as $D=[0,1]^2$ with the source term $f=1$ and the initial value $u_0=x(1-x)y(1-y)$. We consider three representative defects: a well defect, an L-shaped defect, and a porous defect. These geometries are selected to model distinct physical scenarios: for instance, L-shaped defects often arise in problems involving channelized flow, while porous defects are typical of materials with distributed 
inclusions~\cite{oden2000estimation}. We note that, as suggested by the theoretical analysis, the method’s performance may deteriorate when the defect region is excessively large.

The region $K$ is constructed by generating a single-layer mesh surrounding $K_0$. If $K_0$ consists of several disjoint subdomains, each subdomain is treated separately. Consequently, our method is equally applicable to cases with multiple defects. Figure~\ref{K0} illustrates $K_0$ and $K$ for three types of defects. 

The transition function $\rho$ on $K\backslash K_0$ is constructed via linear Lagrange interpolation on the single-layer mesh described above, ensuring continuity and straightforward implementation. This approach, adopted in~\cite{mingNitscheHybridMultiscale2022}, has proven effective in practice. Although Theorem~\ref{thm:local} suggests that the indicator function $\chi_K$ leads to better performance for local $H^1$-estimates, our experiments show that the interpolated transition function performs comparably in local $H^1$-estimates and yields improved results for local 
$L^2$-estimates.

\begin{figure}[!t]
	\centering
	\includegraphics[scale=0.2]{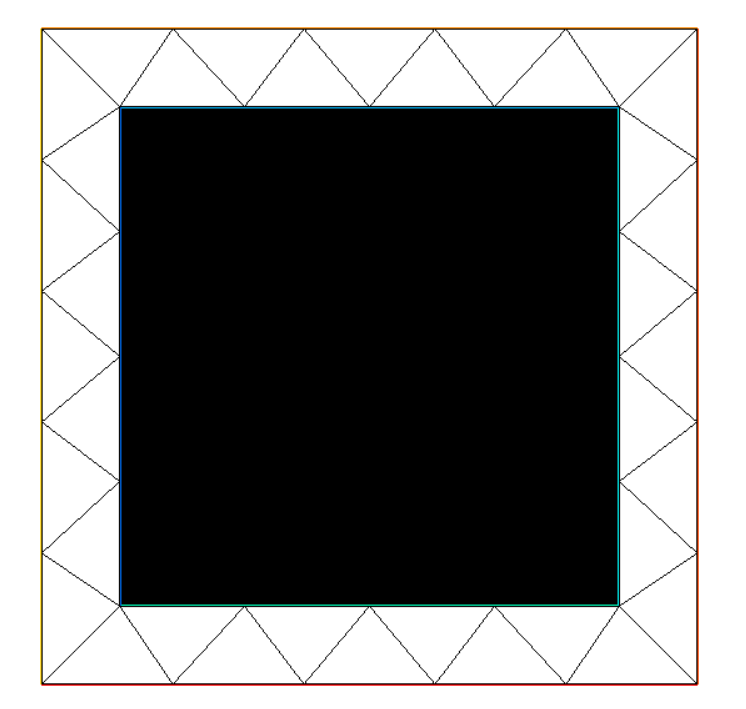}
	\includegraphics[scale=0.2]{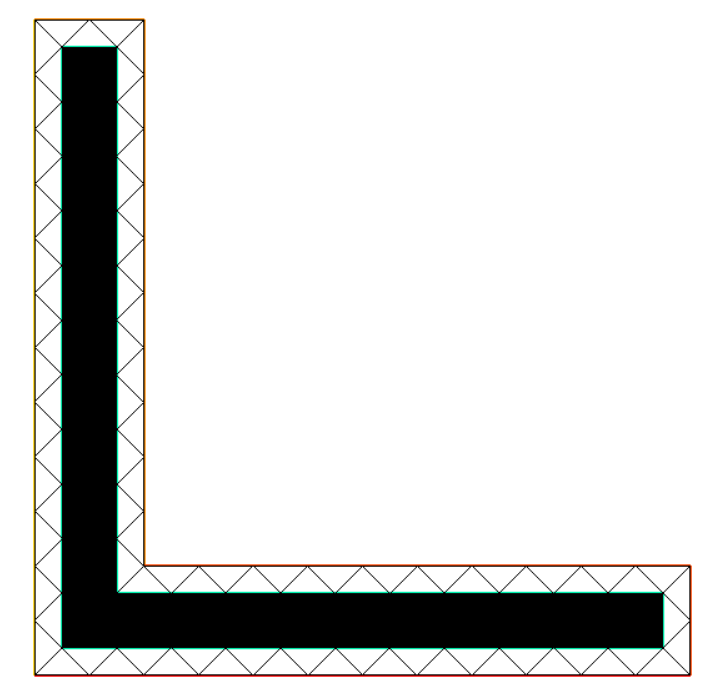}
	\includegraphics[scale=0.2]{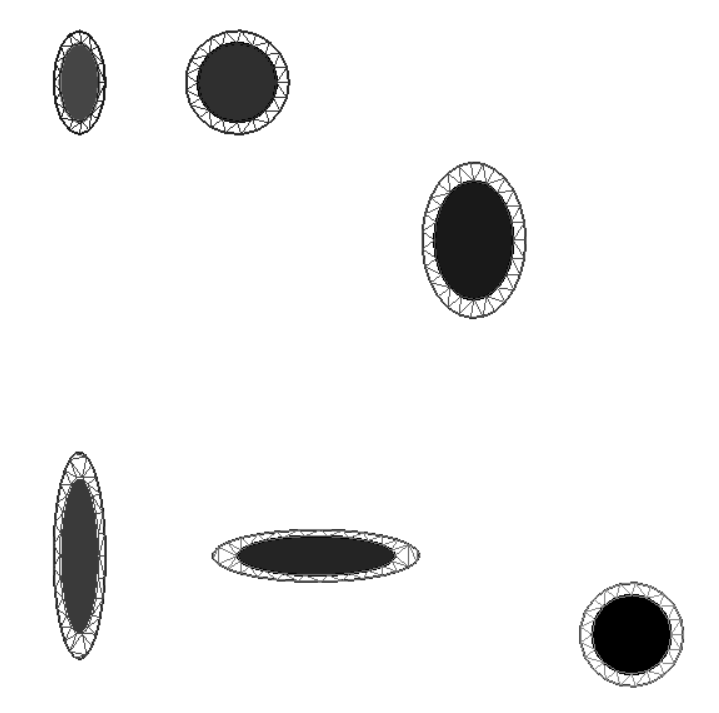}
	\caption{Three defects. Left: well defect. Mid: L-shaped defect. Right: Porous defect}
	\label{K0}
\end{figure}

Relative errors are reported at $T=1$. We consider the global quantities
\[
e_{0}(D\backslash K)=\frac{\Vert U_m-u_m\Vert_{L^2(D\backslash K)}}{\Vert u_m\Vert_{L^2(D\backslash K)}},\qquad
e_{1}(D\backslash K)=\frac{\Vert\nabla (U^m-u_m)\Vert_{L^2(D\backslash K)}}{\Vert \nabla u_m\Vert_{L^2(D\backslash K)}},
\]
and the local quantities
\[
e_{0}(K_0)=\frac{\Vert U_m-u_m^\varepsilon\Vert_{L^2(K_0)}}{\Vert u_m^\varepsilon\Vert_{L^2(K_0)}},
\qquad
e_{1}(K_0)=\frac{\Vert\nabla (U_m-u^\varepsilon_m)\Vert_{L^2(K_0)}}{\Vert \nabla u^\varepsilon_m\Vert_{L^2(K_0)}}
\]
on defect $K_0$, where $m=T/\Delta t$. We fix the time step $\Delta t =0.02$, and vary $h$ and $H$. The reference solution $u^0$ and $u^\varepsilon$ are computed with $h=1/3000$ and $\Delta t=0.01$. All the numerical experiments are conducted on FreeFEM++ \cite{MR3043640}.
\subsection{Example 1: Coefficient with two scales}
The first two-scale diffusion coefficient is taken from \cite{huang_concurrent_2018} with
\begin{equation}\label{eq:aeps}
a^\varepsilon(x) = \frac{(R_1+R_2\sin (2\pi x_1))(R_1+R_2\cos (2\pi x_2))}{(R_1+R_2\sin(2\pi x_1/\varepsilon))(R_1+R_2\sin(2\pi x_2/\varepsilon))}I.
\end{equation}
A direct calculation gives the effective coefficient
\[
A(x)= \frac{(R_1+R_2\sin (2\pi x_1))(R_1+R_2\cos (2\pi x_2))}
	{R_1\sqrt{R_1^2-R_2^2}}I.
\]
Here we select $\varepsilon = 0.01$, $R_1=2.5$ and $R_2=1.5$. We consider the well defect $K_0$ a $0.1\times 0.1$ square, and $K$ a $0.12 \times 0.12$ square in the center of $D$; i.e., $K_0=[0.45,0.55]^2$ and $K=[0.44,0.56]^2$. Therefore in this case, $|K_0|=0.01$ and $|K|=0.0144$.

To investigate the macroscopic error, we fix $\Delta t=0.02$ and $h=1/1000$ while varying $H$. The results, shown in the left columns of Table~\ref{e1wellmac}, correspond to the case $A_H=A$, so that $e(\mathrm{HMM})=0$. We observe that $e_1(D\backslash K)$ converges at first order in $H$. However, the convergence rate of $e_0(D\backslash K)$ drops sharply from approximately $2$ to about $1$, likely due to the combined influence of the defect region $K$ and the time discretization step $\Delta t$ in both the reference and numerical solutions. Overall, the experimental findings align well with the theoretical predictions.
\begin{table}[!t]
	\centering
	\caption{Macroscopic and microscopic errors for two scales coefficient with well defect.}
	\label{e1wellmac}
	\begin{tabular}{lllll|lllll}
		\hline
	$1/H$   & $e_0(D\backslash K)$ & order & $e_1(D\backslash K)$ & order & $1/h$ &  $e_0(K_0)$ & order & $e_1(K_0)$ & order\\ \hline
	\multicolumn{5}{l|}{Macro errors with $h=10^{-3}$} & \multicolumn{5}{l}{Micro errors with $H=10^{-2}$}\\
		20  & 5.54E-03 & ~ & 6.70E-02 & ~ & 200 & 3.72E-03 & ~    & 2.88E-01   & ~       \\
		40  & 1.39E-03 & 1.99  & 3.30E-02 & 1.02 & 400 & 1.07E-03 & 1.80  & 1.44E-01 & 1.00 \\ 
		80  & 4.38E-04 & 1.67  & 1.64E-02 & 1.01 & 800 & 3.42E-04 & 1.63  & 6.26E-02 & 1.20 \\ 
		160 & 2.48E-04 & 0.82  & 8.20E-03 & 1.00 & 1600 & 2.36E-04 & 0.53  & 2.65E-02 & 1.24 \\ 
\hline
	\end{tabular}
\end{table}

As to the microscopic error, we vary $h$ from $1/200$ to $1/1600$ fixing $\Delta t =0.02$ and $H=1/100$. As in the macroscopic case, the homogenization approximation error $e(\mathrm{HMM})=0$. Moreover, since $H^2=10^{-4}\ll h$, the local mesh size $h$ is the dominant error source for $e_1(K_0)$; see Remark \ref{rmk:local}. The results are presented in the right columns of Table~\ref{e1wellmac}, showing that $e_1(K_0)$ exhibits first-order convergence with respect to $h$, while $e_0(K_0)$ tends to a fixed value as $h$ decreases.
\subsection{Example 2: Different transition functions.}
In this example, we investigate the influence of different choices of the transition functions. Since numerical results for $C^1$- and $C^0$-transition functions have been reported in~\cite{mingNitscheHybridMultiscale2022}, we focus here on the discontinuous selections, particularly on the indicator functions. We retain the two-scale diffusion coefficient defined in~\eqref{eq:aeps}, with the well defect given by $K_0=[0.45,0.55]^2$ and $K=[0.44,0.56]^2$.
\begin{table}[!t]
	\centering
	\caption{Macroscopic and microscopic errors for different transition functions.}
	\label{tab:e1rho}
	\begin{tabular}{lllll|lllll}
		\hline
	$1/H$   & $e_0(D\backslash K)$ & order & $e_1(D\backslash K)$ & order & $1/h$ &  $e_0(K_0)$ & order & $e_1(K_0)$ & order\\ \hline
	\multicolumn{5}{l|}{Macro errors with $\rho=\chi_{K_0},h=10^{-3}$} & \multicolumn{5}{l}{Micro errors with $\rho=\chi_{K_0},H=10^{-2}$}\\
		20  & 5.51E-03 & ~ & 6.70E-02 & ~ & 200 & 2.92E-03 & ~ & 2.87E-01   & ~       \\
		40  & 1.36E-03 & 2.02 & 3.29E-02 & 1.03 & 400 & 9.32E-04 & 1.65 & 1.54E-01 & 0.90\\ 
		80  & 3.96E-04 & 1.78 & 1.64E-02 & 1.00 & 800 & 9.55E-04 & -0.04 & 8.45E-02 & 0.87\\ 
		160 & 2.01E-04 & 0.98 & 8.17E-03 & 1.01 & 1600 & 1.06E-03 & -0.15 & 6.21E-02 & 0.44\\  
	\hline
	\multicolumn{5}{l|}{Macro errors with $\rho=\chi_{K},h=10^{-3}$} & \multicolumn{5}{l}{Micro errors with $\rho=\chi_{K},H=10^{-2}$}\\
	20  & 5.52E-03 & ~ & 6.72E-02 & ~ & 200 & 3.54E-03 & ~ & 2.91E-01  & ~       \\
		40  & 1.38E-03 & 2.00 & 3.33E-02 & 1.01 & 400 & 8.58E-04 & 2.04 & 1.48E-01 & 0.98\\ 
		80  & 4.39E-04 & 1.65 & 1.70E-02 & 0.97 & 800 & 8.52E-04 & 0.01 & 6.61E-02 & 1.16\\ 
		160 & 2.55E-04 & 0.78 & 8.41E-03 & 1.01 & 1600 & 1.02E-03 & -0.26 & 2.50E-02 & 1.40\\  
		\hline
	\end{tabular}
\end{table}

We fix $\Delta t=0.02$ and $h=1/1000$ and summarize the macroscopic errors in the left columns of Table \ref{tab:e1rho}. While the results for different transition functions are comparable, $\rho=\chi_{K_0}$ performs slightly better. This observation aligns with Theorem~\ref{thm:global}. Specifically, the factor $\eta(K)$ should more precisely be expressed as $\eta(\mathrm{supp}\,\rho)$. Consequently, when $\rho=\chi_{K_0}$, the area of its support equals to $\abs{K_0}=0.01$, which is slightly smaller than $\abs{K}=0.0144$,  resulting in a smaller contribution from this term.

The behaviour of the microscopic errors is more nuanced, as shown in the right columns of Table \ref{tab:e1rho}. We fix $\Delta t=0.02$ and $H=1/100$. The indicator function $\chi_{K_0}$ yields the poorest performance, consistent with the findings reported in \cite{huang_concurrent_2018}. This can be explained by Theorem \ref{thm:localest1}, which introduces a factor of $\abs{K\backslash K_0}^{1/2-1/p}$ in this case. In contrast, $\chi_K$ performs best for $e_1(K_0)$, in line with Theorem \ref{thm:local}. The linear Lagrange interpolation yields comparable results in practice, as seen in Table~\ref{e1wellmac}. This is reasonable because the principle underlying Theorem \ref{thm:local} is that the coefficient near but outside $K_0$ should remain close to $a^\varepsilon$, a condition that still holds for continuous $\rho$, albeit with mild pollution.

The performance of $e_0(K)$ is poor for the discontinuous transition functions, including both $\chi_{K_0}$ and $\chi_K$. Although a specific theoretical analysis for this case is beyond the scope of the present 
work, this aligns with the general expectation that duality-based $L^2$-estimates typically require the continuity of the coefficient. In view of these observations, we henceforth adopt the linear Lagrange interpolation as the transition function in the subsequent experiments.
\subsection{Example 3: Different sizes of defects.}
In this example, we continue to use the two-scale coefficient given in~\eqref{eq:aeps} and examine the effect of varying the size of the well defect. For macroscopic errors, we fix $\Delta t=0.02$ and $h=1/1000$. For microscopic errors, we set $\Delta t=0.02$ and $H=1/100$. First, we consider a smaller defect $K_0=[0.48,0.52]^2$ and $K=[0.47,0.53]^2$ with corresponding areas $|K_0|=0.0016$ and $|K|=0.0036$, substantially smaller than those in  Table \ref{e1wellmac}. We report the macroscopic and microscopic errors in the top rows of Table \ref{tab:e1size}. The results show a noticeable improvement in the $L^2$‑estimates, confirming that the previously observed reduction in convergence rate was indeed caused by the larger defect $K$. Nevertheless, the errors eventually saturate around $10 
^{-4}$, which is mainly attributable to the fixed time step $\Delta t$ shared by both the reference and the numerical solutions.

Next, we consider a larger defect with $K_0=[0.4,0.6]^2$ and $K=[0.39,0.61]^2$, the corresponding areas are $|K_0|=0.04$ and $|K|=0.0484$, much larger than those in Table~\ref{e1wellmac}. Macroscopic and microscopic errors are shown in the bottom rows of 
Table~\ref{tab:e1size}.
First-order convergence is maintained for both $e_1(D\backslash K)$ and $e_1(K_0)$. Although $e_0(K_0)$
initially appears to converge more rapidly, this impression stems mainly from its lower accuracy on coarse grids. Compared with the smaller well defects, the $L^2$-error estimates deteriorate, highlighting the stronger influence of a larger defect region.
\begin{table}[!t]
	\centering
	\caption{Macroscopic and microscopic errors for different sizes of defects.}
	\label{tab:e1size}
	\begin{tabular}{lllll|lllll}
		\hline
	$1/H$   & $e_0(D\backslash K)$ & order & $e_1(D\backslash K)$ & order & $1/h$ &  $e_0(K_0)$ & order & $e_1(K_0)$ & order\\ \hline
	\multicolumn{5}{l|}{Macro error with $K=[0.47,0.53]^2,h=10^{-3}$} & \multicolumn{5}{l}{Micro errors with $K_0=[0.48,0.52]^2,H=10^{-2}$}\\
		20  & 5.61E-03 &~ & 6.81E-02 & ~ & 200 & 1.01E-03 & ~& 2.54E-01      \\
		40  & 1.39E-03 & 2.01 & 3.32E-02 & 1.04 & 400 & 3.79E-04 & 1.41 & 1.62E-01 & 0.65\\ 
		80  & 3.78E-04 & 1.88 & 1.61E-02 & 1.04 & 800 & 1.42E-04 & 1.42 & 8.60E-02 & 0.91\\ 
		160 & 1.42E-04 & 1.41 & 7.57E-03 & 1.09 & 1600 & 1.46E-04 & -0.04 & 4.18E-02 & 1.04\\  
	\hline
	\multicolumn{5}{l|}{Macro error with $K=[0.39,0.61]^2,h=10^{-3}$} & \multicolumn{5}{l}{Micro errors with $K_0=[0.4,0.6]^2,H=10^{-2}$}\\
	20  & 5.53E-03 & ~ & 6.38E-02 & ~ & 200 & 9.59E-03 & ~ & 2.85E-01      \\
		40  & 1.53E-03 & 1.85 & 3.12E-02 & 1.03 & 400 & 3.06E-03 & 1.65 & 1.36E-01 & 1.07\\ 
		80  & 6.29E-04 & 1.28 & 1.52E-02 & 1.04 & 800 & 9.65E-04 & 1.67 & 4.94E-02 &1.46\\ 
		160 & 4.67E-04 & 0.43 & 7.14E-03 & 1.09 & 1600 & 5.87E-04 & 0.72 & 2.51E-02 & 0.98\\
		\hline
	\end{tabular}
\end{table}
\subsection{Example 4: Different shapes of defects.}
In this example, we continue to employ the two-scale coefficient given in~\eqref{eq:aeps} and consider different defect geometries, as shown in Figure~\ref{K0}. For the L-shaped defect $K_0$, the vertices (listed counterclockwise) are 
\[
(0.4, 0.4),\; (0.76, 0.4), \;(0.76, 0.41), \;(0.41, 0.41), \;(0.41, 0.76) \text{\; and\;} (0.4, 0.76). 
\]
The longest edge of $K_0$ measures $0.36$ and the shortest $0.01$, resulting in an aspect ratio of $36$. This configuration can be used to model a thin channel. 

The annular layer between $K$ and $K_0$ has thickness $0.005$,  so the vertices of $K$ are 
\[
(0.395, 0.395), \;(0.765, 0.395), \;(0.765, 0.415),\; (0.415, 0.415), (0.415, 0.765)\text{\; and\;} (0.395, 0.765). 
\]
For $K$, the longest edge is $0.37$ and the shortest $0.02$.

The macroscopic and microscopic errors are reported in the top rows of Table~\ref{tab:e1shape}. Following the setup of Example 1, we fix $\Delta t=0.02$ and $h=1/1000$ for the macroscopic errors, and $\Delta t=0.02$ and $H=1/100$ for the microscopic errors. The areas $|K_0|=0.0071$ and $|K|=0.0144$ are closed to those of the well defect in Example 1, thereby reducing size effects. The defect shape appears to primarily affect the $L^2$-estimates, specifically $e_0(D\backslash K)$ and $e_0(K_0)$, while first-order convergence is still maintained for both $e_1(D\backslash K)$ and $e_1(K_0)$.
\begin{table}[!t]
	\centering
	\caption{Macroscopic and microscopic errors for different shapes of defects.}
	\label{tab:e1shape}
	\begin{tabular}{lllll|lllll}
		\hline
	$1/H$   & $e_0(D\backslash K)$ & order & $e_1(D\backslash K)$ & order & $1/h$ &  $e_0(K_0)$ & order & $e_1(K_0)$ & order\\ \hline
	\multicolumn{5}{l|}{Macro errors with L-shaped defect, $h=10^{-3}$} & \multicolumn{5}{l}{Micro errors with L-shaped defect, $H=10^{-2}$}\\
		20  & 5.27E-03 & ~ & 6.22E-02 & ~ & 200 & 1.96E-03 & ~ & 1.99E-01  & ~       \\
		40  & 1.52E-03 & 1.79 & 3.04E-02 & 1.03 & 400 & 1.33E-03 & 0.56 & 1.65E-01 & 0.27\\ 
		80  & 5.34E-04 & 1.41 & 1.54E-02 & 0.98 & 800 & 9.27E-04 & 0.52 & 7.75E-02 & 1.09\\ 
		160 & 3.76E-04 & 0.51 & 8.12E-03 & 0.92 & 1600 & 8.65E-04 & 0.10 & 3.94E-02 & 0.98\\  
	\hline
	\multicolumn{5}{l|}{Macro errors with porous defect, $h=H/2$} & \multicolumn{5}{l}{Micro errors with porous defect, $H=2h$}\\
	20  & 5.97E-03 & ~ & 6.17E-02 & ~ & 200 & 4.05E-03 & ~ & 3.23E-01   & ~       \\
		40  & 3.53E-03 & 0.76 & 3.07E-02 & 1.01 & 400 & 2.46E-03 & 0.72 & 1.86E-01 & 0.80\\ 
		80  & 2.46E-03 & 0.52 & 1.60E-02 & 0.94 & 800 & 2.39E-03 & 0.04 & 9.58E-02 & 0.96\\ 
		160 & 1.02E-03 & 1.27 & 7.70E-03 & 1.06 & 1600 &  2.00E-03 & 0.26 & 5.35E-02 & 0.84\\  
		\hline
	\end{tabular}
\end{table}

The following example  illustrates the capability of the proposed method to handle multiple defects. The porous defect $K_0$ comprises six ellipses centered at 
\[
(0.2, 0.8), \;(0.2, 0.2),\; (0.4, 0.8),\; (0.5, 0.2),\; (0.7, 0.6),\; (0.9, 0.1). 
\]
The semi-axes in the $x$-and $y$-directions are, respectively, 
\[
(R/2, R),\; (R/2,2R),\; (R,R), \;(2R,R/2), \;(R,3R/2), \;(R,R)
\]
with $R=0.023$. This configuration can be used to model materials with different type of inclusions. The outer set $K$ consists of ellipses with the same centers, but with $R=0.028$. Consequently, $|K_0|\approx 0.00997$ and $|K|\approx 0.01478$. 

Because the porous defect induces very thin elements near the pores and much coarser elements elsewhere, resulting in mesh with poor element, we maintain a fixed ratio between coarse and fine mesh sizes by setting $H=2h$ in this example.  To mitigate the computational cost associated with such meshes, one could employ non-matching grids and enforce interface conditions via Nitsche’s method, as in \cite{mingNitscheHybridMultiscale2022}, but this lies beyond the scope of the present work. We report the macroscopic and microscopic errors in the bottom rows of Table~\ref{tab:e1shape} and observe the performance is consistent with that of the L-shaped defect.
\subsection{Example 5: Coefficient without scale separation}
In previous examples, there is scale separation in the diffusion coefficient. This example is taken from \cite{Abdulle2015AnOH}. The coefficient
$a^{\varepsilon}=\chi_{K_0} \tilde{a}+\left(1-\chi_{K_0}\right) \tilde{a}^{\varepsilon}$, where
\[
	\tilde{a}(x)=3+\frac{1}{7} \sum_{j=0}^4 \sum_{i=1}^j \frac{1}{j+1}
	\cos \left(\left\lfloor 8\left(i x_2-\frac{x_1}{i+1}\right)\right\rfloor
	+\left\lfloor 150 i x_1\right\rfloor+\left\lfloor 150 x_2\right\rfloor\right),
\]
and
\[
	\tilde{a}^{\varepsilon}(x)=2.1+\cos \left(2 \pi x_1 / \varepsilon\right) \cos
	\left(2 \pi x_2 / \varepsilon\right)+\sin (4 x_1^2 x_2^2).
\]
We still consider the well defect with $K_0=[0.45,0.55]^2$ and $K=[0.44,0.56]^2$, as in Example 1, and set $\varepsilon = 0.0063$ for the experiment. The effective matrix $A_H$ is obtained using the online-offline method proposed 
in~\cite{huang_efficient_2020}.

We set $\Delta t=0.02$ and $h=1/1000$ and report the macroscopic errors in the left columns
of Table~\ref{e2wellmac}. The results are similar to those in Example 1. Let $\Delta t=0.02$ and $H=1/100$, the microscopic errors are shown in the right columns of Table~\ref{e2wellmac}. The quantity $e_0(K_0)$ remains nearly constant at approximately $5\times 10^{-4}$, a behavior attributable to the discontinuity of $\tilde{a}$ as in Example 2. Meanwhile, the convergent rate for $e_1(K_0)$ stays below $1$, reflecting the reduced regularity of $\tilde{a}$ caused by the roughness of $u^\varepsilon$.
\begin{table}[!t]
	\centering
	\caption{Macroscopic and microscopic errors for no scale separation coefficient with well defect.}
	\label{e2wellmac}
    \begin{tabular}{lllll|lllll}
		\hline
		$1/H$  & $e_0(D \backslash K)$ & order & $e_1(D \backslash K)$ & order & $1/h$ &  $e_0(K_0)$ & order & $e_1(K_0)$ & order\\ \hline
		\multicolumn{5}{l|}{Macro errors with $h=10^{-3}$} & \multicolumn{5}{l}{Micro errors with $H=10^{-2}$}\\
		20 & 5.15E-03 & ~ & 7.19E-02 & ~ & 200 & 5.02E-04 & ~ & 4.23E-02 & ~\\
        40 & 1.18E-03 & 2.13  & 3.42E-02 & 1.07 & 400 & 5.22E-04 & -0.06  & 3.20E-02 & 0.40 \\ 
        80 & 3.14E-04 & 1.91  & 1.69E-02 & 1.02 & 800 & 5.43E-04 & -0.06  & 1.87E-02 & 0.78 \\ 
        160 & 1.16E-04 & 1.43  & 8.05E-03 & 1.07 & 1600 & 5.51E-04 & -0.02  & 1.08E-02 & 0.79 \\ 
\hline
	\end{tabular}
\end{table}
\section{Conclusion}\label{sec:concl}
In this paper, we presented a concurrent global-local numerical method for solving multiscale parabolic equations in divergence form. The proposed method effectively integrates global and local information to accurately capture both macroscopic and microscopic behaviors of the solution, which has been confirmed both theoretically and experimentally. For multiscale parabolic equations with time-independent coefficient, we successfully eliminate the factor $\Delta t^{-1/2}$ for both macroscopic and microscopic accuracy. It is also interesting to study such problems with time-dependent coefficients, time varying defects and time varying boundary conditions.
\bibliography{Reference}
\end{document}